\documentclass[11pt,a4paper]{article}
\usepackage[utf8]{inputenc}
\usepackage{authblk}
\title{\large\textbf{Local exact controllability to constant trajectories for Navier-Stokes-Korteweg model}\footnote{The author is partially supported by the Project TRECOS ANR-20-CE40-0009 funded by the ANR.}}
\author{\small Adrien \textsc{Tendani Soler}\footnote{\textit{E-mail}: \texttt{adrien.tendani-soler@math.u-bordeaux.fr}}}
\affil{\scriptsize Institut de Mathématiques de Bordeaux UMR 5251,\\ Université de Bordeaux, Bordeaux INP, CNRS\\
F-33405 Talence, France}
\usepackage[nottoc]{tocbibind}
\usepackage[T1]{fontenc}
\usepackage{amsmath}
\usepackage{minitoc}
\usepackage{graphicx}
\usepackage{amsthm}
\usepackage{amsfonts}
\usepackage{amssymb}
\usepackage{xcolor}
\usepackage{setspace}
\usepackage{varioref}
\usepackage{color}
\usepackage{makeidx}
\usepackage
{geometry}
\usepackage{varioref}
\usepackage{fullpage}
\usepackage{hyperref}
\usepackage{cleveref}
\hypersetup{
    colorlinks=true,
    linkcolor=blue,
    filecolor=magenta,     
    urlcolor=cyan,
    pdfpagemode=FullScreen,
    }
\usepackage{appendix}
\newtheorem{theo}{Theorem}[section]
\newtheorem{rem}[theo]{Remark}

\newtheorem*{notation}{Notation}
\newtheorem{lem}[theo]{Lemma}

\newtheorem{coro}[theo]{Corollary}

\numberwithin{equation}{section}

\newcommand{\R}{\mathbb{R}}
\newcommand{\Z}{\mathbb{Z}}

\newcommand{\vphi}{\varphi}

\newcommand{\T}{\mathbb{T}}

\newcommand{\enstq}[2]{\left\{#1~\middle|~#2\right\}}

\DeclareMathOperator{\divergence}{\mathop{}div}
\newcommand*\Laplace{\mathop{}\!\mathbin\bigtriangleup}

\usepackage{abstract}

\newcommand{\revdots}{\mathinner
{ \mkern1mu\raise1pt\vbox{\kern7pt\hbox{.}} \mkern2mu\raise4pt\hbox{.} \mkern2mu\raise7pt\hbox{.}\mkern1mu}
}

\newcommand{\croixdots}{ \mathinner{
\mkern1mu\raise7pt\vbox{\kern7pt 
\hbox{.}}
\mkern-5mu
\raise7pt\vbox{\hbox{.}}
\mkern1mu }
}

\renewcommand{\leq}{\leqslant}
\renewcommand{\geq}{\geqslant}

\renewcommand{\abstitlestyle}[1]{\noindent}
\date{\today} 

\begin{document}
\newgeometry{twoside, inner=2cm,outer=2cm,top=2.23cm,bottom=3.1cm}

{\setlength{\baselineskip}{0.1\baselineskip}
\maketitle}

\begin{abstract}
In this article, we study the local exact controllability to a constant trajectory for a compressible Navier-Stokes-Korteweg system on the torus in dimension $ d\in\{1,2,3\}$ when the control acts on an open subset. To be more precise, we obtain the local exact controllability to the constant state $(\rho_{\star},0)$ for arbitrary small positive times and without any geometric condition on the control region. In order to do so, we analyze the control properties of the linearized equation, and present a detailed study of the observability of the adjoint equations. In particular, we shall exhibit the parabolic (possibly also dispersive) structure of these adjoint equations. Based on that, we will be able to recover observability of the adjoint system through Carleman estimates. 
\end{abstract}

\hfill
{\setlength{\baselineskip}{0.95\baselineskip}
\scriptsize\tableofcontents\par}
\hfill
\section{Introduction}
 In this work, we are interested in the control properties of the Navier-Stokes-Korteweg system. This system describes a compressible and viscous fluid on a region of $\R^d$ with $d\in\{1,2,3\}$, of density $\rho=\rho(t,x)$ and velocity field $u=u(t,x)$ and reads
\begin{equation}
\label{eqNSK}
\begin{cases}
\partial_t\rho+\divergence(\rho u)=0,\\
\partial_t(\rho u)+\divergence(\rho u\otimes u)-{\mathcal A(\rho)} u+\nabla (P(\rho))=\divergence(\mathcal{K(\rho)}),\\
\end{cases}
\end{equation}
where $P(\rho)$ is the pressure function assumed to depend only on the density $\rho$, $${\mathcal A(\rho)}u:=\divergence\left(2\mu(\rho)D_S(u)\right)+\nabla\left(\nu(\rho)\divergence{(u)}\right)$$ is the viscosity part, $D_S(u):=\frac 12(\nabla u+\!^t\nabla u)$ is the symmetric gradient and the capillarity tensor $\mathcal{K(\rho)}$ is given by 
$$
{\mathcal K(\rho)}:=\rho \divergence(\kappa(\rho)\nabla\rho)I_{\R^d}
+\frac12\big(\kappa(\rho)-\rho\kappa'(\rho)\big)|\nabla\rho|^2I_{\R^d}
-\kappa(\rho)\nabla\rho\otimes\nabla\rho,
$$
where $I_{\R^d}$ denotes the $d\times d$ identity matrix, see \cite{StructureofKortewegmodelsandstabilityofdiffuseinterfaces}. The coefficients $\nu=\nu(\rho)$ and $\mu=\mu(\rho)$ designate the bulk and shear viscosity, respectively and $\kappa=\kappa(\rho)$ the capillarity function. Note that all these coefficients are assumed to be functions of the density. 

The Navier-Stokes-Korteweg system describes a two-phase compressible and viscous fluid in the case of diffuse interface, in which the change of phase corresponds to a fast but regular transition zone for the density and the velocity. We refer to \cite{Sharpanddiffuseinterfacemethodsforphasetransitionproblemsinliquidvapourflows} for the modeling of phase transition and to \cite{OntheThermodynamicsofInterstitialWorking} for the full derivation of the Navier-Stokes-Korteweg model. Note that this model includes as a special case the quantum Navier-Stokes equation see for instance \cite{DerivationofviscouscorrectiontermsfortheisothermalquantumEulermodel} for its derivation from the Wigner equation.

 Based on physical considerations, it is natural to assume that 
\begin{equation}
	\label{Physical-Assumptions}
	\nu \geq 0, \qquad 2 \mu + \nu \geq 0, \qquad \kappa \geq 0.
\end{equation}
Note that if the viscosity parameters satisfy $ \mu > 0$ and $\mu + \nu >0$, and the capillarity coefficient $\kappa$ vanishes, then System \eqref{eqNSK} coincides with the compressible Navier-Stokes system. Below, we will be focusing on the case $\mu>0$, $2\mu+\nu>0$ and $\kappa>0$, at least locally, which corresponds to the Navier-Stokes-Korteweg model.
\medskip

{\bf Main results.} Before going further, let us remark that system \eqref{eqNSK} possesses some specific stationary states given by constant states $(\rho_\star, u_\star)$ with $\rho_\star >0$ and $u_\star \in \R^d$. Our goal is to analyse the local exact controllability property of \eqref{eqNSK} around these constant states $(\rho_\star, u_\star)$. For simplicity, we will reduce our analysis only to the case $u_\star = 0$ (the case $u_\star \in \R^d$ can be handled similarly).

Let us describe the geometrical setting. We work in the $d$-dimensional torus $\T_L:=(\R/L\Z)^{d}$ identified with $[0,L]^d$ with periodic boundary conditions, where $L>0$, and the controls will be assumed to act on some non-empty open subset $\omega$ of $\T_L$. 

Our main result is the following one: 

\begin{theo}
\label{Controle intern de NSK quantique sur le toremain}
	Let $d\in\{1,2,3\}$, $L>0$, and $\omega$ be a non-empty open subset of $\T_L$. 
	
	Let $\rho_{\star}>0$, and let us assume that 
	\begin{enumerate}
	    \item[{\bf (H1)}]\label{H1} $\kappa(\rho_{\star})$,  $\mu(\rho_{\star})$ and $2\mu(\rho_{\star})+\nu(\rho_{\star})$ are positive;
	    \item[{\bf (H2)}]\label{H2} there exists $\eta \in (0,\rho_\star)$ such that $\mu$ and $\nu$ belong to $\mathcal{C}^{2}([-\eta+\rho_{\star},\eta+\rho_{\star}])$ and $P$ and $\kappa$ belong to $\mathcal{C}^{3}([-\eta+\rho_{\star},\eta+\rho_{\star}])$.
	\end{enumerate}
	
	Then there exists $\delta>0$ such that, for all $(\rho_0,u_0)\in H^2(\T_L)\times H^1(\T_L)$ satisfying 
\begin{equation*}
    \|(\rho_0-\rho_{\star},u_0)\|_{H^2(\T_L)\times H^1(\T_L)}\leq \delta,
\end{equation*}
there exist a control $(v_{\rho},v_u)\in L^2(0,T;H^2(\T_L))\times L^2(0,T;H^1(\T_L))$ and a corresponding controlled trajectory $(\rho,u)$ solving 
\begin{equation}
\label{eqNSKTor}
\begin{cases}
\partial_t\rho+\divergence(\rho u)=v_{\rho}\mathbf{1}_{\omega} & \text{in}\ \ (0,T)\times\T_L,\\
\partial_t(\rho u)+\divergence(\rho u\otimes u)-{\mathcal A(\rho)} u+\nabla (P(\rho))=\divergence(\mathcal{K(\rho)})+v_{u}\mathbf{1}_{\omega} & \text{in}\ \ (0,T)\times\T_L,\\
(\rho,u)_{|_{t=0}}=(\rho_0,u_0) & \text{in } \T_L,
\end{cases}
\end{equation}
and satisfying
$$
	(\rho,u)_{|_{t=T}}=(\rho_{\star},0)\ \ \ \ \textit{in}\ \ \T_L.
$$ 
Besides, the controlled trajectory $(\rho,u)$ enjoys the following regularity
\begin{align*}
	& \rho \in \mathcal{C}([0,T]; H^2(\T_L))\cap L^2(0,T;H^3(\T_L))\cap H^1(0,T;H^1(\T_L)),
	\\
	& u \in \mathcal{C}([0,T]; H^1(\T_L))\cap L^2(0,T;H^2(\T_L))\cap H^1(0,T;L^2(\T_L)), 
\end{align*}
and the following positivity condition
$$
	\inf_{(t,x) \in [0,T] \times \T_L} \rho(t,x) > 0.
$$

\end{theo}

\begin{rem}
	Hypothesis {\bf (H1)} concerning the sign of $\kappa$, $\mu$ and $\nu$ at the point $\rho_{\star}$ guarantees the parabolic-type structure of the linearized system, see afterwards. Let us point out that, here, we work with strong solutions on bounded intervals, so that, we do not need any assumption on the monotonicity of the pressure $P$. 
	
	Hypothesis {\bf (H2)} concerning the regularity of $\mu$, $\nu$, $\kappa$ and $P$ is mainly technical and is needed to handle the non-linear terms in the proof of Theorem \ref{Controle intern de NSK quantique sur le toremain}. 
\end{rem}

Since the torus is a rather academic example, let us start by pointing out that this result leads to an exact controllability result to constant trajectories $(\rho_\star, 0)$ in bounded domains $\Omega$ when the controls act on the whole boundary $\partial \Omega$. To be more precise, we have the following immediate corollary: 
\begin{coro}
\label{main resultmain}
	Let $d\in\{1,2,3\}$, $\Omega$ be a smooth bounded domain of $\R^d$. Let $\rho_{\star}>0$, and let us assume conditions {\bf (H1)} and {\bf (H2)} of Theorem \ref{Controle intern de NSK quantique sur le toremain}.
	
	Then there exists $\delta>0$ such that, for all $(\rho_0,u_0)\in H^2(\Omega)\times H^1(\Omega)$ satisfying 
	\begin{equation}
		\label{condition de petitesse pour donne initial}
		\|(\rho_0-\rho_{\star},u_0)\|_{H^2(\Omega)\times H^1(\Omega)}\leq\delta,
	\end{equation}
	there exists a controlled trajectory $(\rho,u)$ solving 
\begin{equation}
\label{eqNSKOmega}
\begin{cases}
\partial_t\rho+\divergence(\rho u)=0 & \text{in}\ \ (0,T)\times\Omega,\\
\partial_t(\rho u)+\divergence(\rho u\otimes u)-{\mathcal A(\rho)} u+\nabla (P(\rho))=\divergence(\mathcal{K(\rho)})& \text{in}\ \ (0,T)\times\Omega,\\
(\rho,u)_{|_{t=0}}=(\rho_0,u_0) & \text{in } \Omega,
\end{cases}
\end{equation}
and satisfying
$$
	(\rho,u)_{|_{t=T}}=(\rho_{\star},0)\ \ \ \ \textit{in}\ \ \Omega. 
$$ 
Besides, the controlled trajectory $(\rho,u)$ enjoys the following regularity
\begin{align*}
	& \rho \in \mathcal{C}([0,T]; H^2(\Omega))\cap L^2(0,T;H^3(\Omega))\cap H^1(0,T;H^1(\Omega)),
	\\
	& u \in \mathcal{C}([0,T]; H^1(\Omega))\cap L^2(0,T;H^2(\Omega))\cap H^1(0,T;L^2(\Omega)).
\end{align*}
and the following positivity condition
$$
	\inf_{(t,x) \in [0,T] \times \overline\Omega} \rho(t,x) > 0.
$$
\end{coro}

Let us note that the controls do not appear explicitly in the equation \eqref{eqNSKOmega}. In fact, they are hidden in the boundary conditions, which do not appear in \eqref{eqNSKOmega}. 

We will not give the complete details of the proof of Corollary \ref{main resultmain}, and we only sketch it hereafter. Since $\Omega$ is bounded, it can be embedded into some torus $\T_L$, where $\T_L$ is identified to $[0,L]^d$ with periodic condition. We then consider the control problem on the torus $\T_L$ with controls appearing as source terms supported in $\T_L\setminus \overline{\Omega}$, starting from an initial datum obtained as an extension of $(\rho_0, u_0)$ to $\T_L$ (this can be done continuously from $H^2(\Omega) \times H^1(\Omega)$ to $H^2(\T_L) \times H^1(\T_L)$ if $\Omega$ is $\mathcal{C}^2$). Theorem \ref{Controle intern de NSK quantique sur le toremain} gives the existence of a controlled trajectory solving \eqref{eqNSKTor} on the torus, which, by restriction to $\Omega$, gives a solution to \eqref{eqNSKOmega}. 
Before going further, let us emphasize that the control result presented in Theorem \ref{Controle intern de NSK quantique sur le toremain} holds whatever $\omega$ is and in any positive time. This is in sharp contrast with the control results for the linearized compressible Navier-Stokes equation, corresponding to $\kappa = 0$. Indeed, for those equations, the linearized models have some initial data which cannot be controlled in short time, as a result of the transport-parabolic structure of the linearized compressible Navier-Stokes equations, see in particular the works \cite{ControllabilityandStabilizabilityoftheLinearizedCompressibleNavier-StokesSysteminOneDimension} and \cite{Somecontrollabilityresultsforlinearizedcompressiblenavierstokessystem}. However, one can obtain local exact controllability result around trajectories $(\rho_\star,u_\star)$ provided suitable geometric conditions related to the flow of the target velocity field are satisfied, see for instance \cite{LocalexactcontrollabilityforthetwoandthreedimensionalcompressibleNavierStokesequations}, \cite{LocalExactControllabilityfortheOneDimensionalCompressibleNavierStokesEquation} and \cite{Localboundarycontrollabilitytotrajectoriesforthe1DcompressibleNavierStokesequations}.

Let us now explain the reason for this striking difference between the control properties for the compressible Navier-Stokes equations and the Navier-Stokes-Korteweg system. The crucial point is to notice that the Navier-Stokes-Korteweg system actually behaves like a parabolic system. In order to see this property, let us consider the linearized version of \eqref{eqNSK} around $(\rho_\star,0)$, which is
\begin{equation*}
\left \{
\begin{array}{ll}
     \partial_t a +\divergence(u)=0\! 
     &\text{in}\ (0,T)\times\T_L, 
     \\ 
      \partial_t u-\rho_{\star}^{-1}(\mu(\rho_\star) \Laplace{u}-(\mu(\rho_\star)+\nu(\rho_\star))\nabla\divergence (u)) +P'(\rho_\star)\nabla a-\rho_{\star}\kappa(\rho_\star)\nabla\Laplace{a}=0 
      &\text{in}\ (0,T)\times\T_{L},
\end{array}
\right.
\end{equation*}
where $a:=\rho/\rho_{\star}-1$. To simplify the argument we omit the term $P'(\rho_\star)\nabla a$ in the second equation, since it is of lower order compared to $\rho_{\star}\kappa(\rho_\star)\nabla\Laplace{a}$.  By taking the Laplacian of the first equation and the divergence of the second equation, setting $b = -\Delta a$ and $q:=\divergence(u)$, we obtain the following subsystem of two scalar equations
\begin{equation}
\label{equation sur Pu et Qu}
\left \{
\begin{array}{lcr}

     \partial_t b - \Delta q=0\ \qquad\qquad\qquad\ \  \ \ \ \ \ \ \  \ \ \ \ \ \ \ \ \ \ \ \ \ \ \ &\text{in}\ (0,T)\times\T_{L},\\
      \partial_t q-\rho_{\star}^{-1}(2\mu(\rho_\star)+\nu(\rho_\star))\Laplace q+\rho_{\star}\kappa(\rho_\star)\Laplace b=0 &\text{in}\ (0,T)\times\T_{L}.

\end{array}
\right.
\end{equation}

Accordingly, it is essential to analyze the structure of the matrix
\begin{equation}
\label{matrice A}
   A:= \begin{pmatrix}
	 0 & -1 \\
	\rho_{\star}\kappa(\rho_\star) & - \rho_{\star}^{-1}(2\mu(\rho_\star)+\nu(\rho_\star))
\end{pmatrix}.
\end{equation}
Explicit computations show that:
\begin{enumerate}
\item If $\rho_{\star}^{-2}(2\mu(\rho_\star)+\nu(\rho_\star))^2> 4\rho_{\star}\kappa(\rho_\star)$, then $A$  is diagonalizable and has two real eigenvalues:
$$
\frac{-\rho_{\star}^{-1}(2\mu(\rho_\star)+\nu(\rho_\star))\pm\sqrt{\rho_{\star}^{-2}(2\mu(\rho_\star)+\nu(\rho_\star))^2-4\rho_{\star}\kappa(\rho_\star)}}{2}.
$$
\item If $\rho_{\star}^{-2}(2\mu(\rho_\star)+\nu(\rho_\star))^2<4\rho_{\star}\kappa(\rho_\star)$, then $A$ is diagonalizable and has two complex eigenvalues:
$$
\frac{-\rho_{\star}^{-1}(2\mu(\rho_\star)+\nu(\rho_\star))\pm i\sqrt{4\rho_{\star}\kappa(\rho_\star)-\rho_{\star}^{-2}(2\mu(\rho_\star)+\nu(\rho_\star))^2}}{2}.
$$
\item If $\rho_{\star}^{-2}(2\mu(\rho_\star)+\nu(\rho_\star))^2 = 4\rho_{\star}\kappa(\rho_\star)$, then $A$ is similar to an upper triangular matrix having the double eigenvalue 
$$
	\frac{-\rho_{\star}^{-1}(2\mu(\rho_\star)+\nu(\rho_\star))}{2}.
$$
\end{enumerate}
Assumption {\bf (H1)} implies that, in any of the three above cases, the real parts of the eigenvalues are negative. This therefore implies that the system \eqref{equation sur Pu et Qu} has a parabolic/dispersive structure. 

Note that this parabolic/dispersive structure has already been pointed out in the literature. Indeed, the analytic smoothing effect in space variable for this equation for both the velocity field and the density has been shown in \cite{GevreyanalyticityanddecayforthecompressibleNavier-Stokessystemwithcapillarity} (see also \cite{AnalyticregularityforNavierStokesKortewegmodelonpseudomeasurespaces} and \cite{GlobalexistenceandanalyticityofLpsolutionstothecompressiblefluidmodelofKortewegtype}) and is precisely based on  the derivation of dissipative estimate on the Fourier modes of the solutions to the linearized system, underlying the parabolic structure of the system. In views of the terminology in \cite{Dissipativestructureforsymmetrichyperbolic-parabolicsystemswithKorteweg-typedispersion}, we say that the system is \textit{purely parabolic} in the cases 1 and 3. In the case 2, we have dispersion in addition to the dissipation (note that this case contains the quantum Navier-Stokes system, see for instance \cite{OnNavierStokesKortewegandEulerKortewegsystemsapplicationtoquantumfluidsmodels}) and the system should thus be considered as \textit{parabolic/dispersive}.\\ 
Our analysis will be based on a similar discussion, but on the adjoint of the linearized equations of \eqref{eqNSKOmega} which involves the transpose matrix $^t\!A$. In fact, our analysis will be split into two cases: \begin{itemize}
    \item $^t\!A$ is diagonalizable which correspond to Items $1$ and $2$ above;
 \item $^t\!A$ is nondiagonalizable which correspond to Item $3$ above.
\end{itemize}

We end this section by mentioning some related results and open problems.

{\bf Cauchy theory.} Although we will not use any result on the Cauchy theory of the Navier-Stokes-Korteweg system, let us point out that the existence of strong solutions, respectively weak,  has been obtained in \cite{GlobalstrangsolutionfortheKortewegsystemwithquantumpressureindimensionNgeq2,GlobalSolutionsofaHighDimensionalSystemforKortewegMaterials}, respectively \cite{GlobalWeakSolutionstoCompressibleNavier-StokesEquationsforQuantumFluids}. Also note that, in the case of the whole space by using Fourier analysis methods related to the parabolic structure mentioned above, the well-posedness of the Cauchy problem in critical Besov spaces for global and local solutions and in dimension $d\geq 2$ is established in \cite{ExistenceofsolutionsforcompressiblefluidmodelsofKortewegtype}.

{\bf Open problems.} In Theorem \ref{Controle intern de NSK quantique sur le toremain}, we consider internal control which appear both in the continuity equations and the momentum equations. In order to get a more physically relevant interpretation of the internal controllability as forces, it would be interesting to investigate the controllability of the Navier-Stokes-Korteweg system when the control acts only in the momentum equations. In such case, one could rely for instance on the algebraic solvability method as in \cite{LocalnullcontrollabilityofthethreedimensionalNavierStokessystemwithadistributedcontrolhavingtwovanishingcomponents, Indirectcontrollabilityofsomelinearparabolicsystemsofmequationswithm1controlsinvolvingcouplingtermsofzeroorfirstorde}. Note that such results have also been obtained in different contexts: for $1$-d compressible Navier-Stokes equation (see for instance \cite{Localboundarycontrollabilitytotrajectoriesforthe1DcompressibleNavierStokesequations}), for coupled hyperbolic-parabolic system (see for instance \cite{Nullcontrollabilityfortheparabolicequationwithacomplexprincipalpart}) and for Kuramoto-Sivashinsky system (see for instance \cite{LocalcontrollabilityofthestabilizedKuramotoSivashinskysystembyasinglecontrolactingontheheatequation}). 


Another  interesting question concerns the controllability of the Navier-Stokes-Korteweg system  \eqref{eqNSK} on bounded open domain of $\R^d$ with controls localized on an nonempty open subset of the boundary. The additional difficulty compared to the ones in Theorem \ref{Controle intern de NSK quantique sur le toremain} is that the algebraic manipulations used to make explicitly appear the parabolic structure of the system create intricate boundary conditions on the systems, which are difficult to handle. We point out that such intricate terms also appear, with a different coupling, in the context of non-homogeneous incompressible Navier-Stokes system, in which they can be handled through appropriate weighted energy estimates, see \cite{Localcontrollabilitytotrajectoriesfornon-homogeneousincompressibleNavier-Stokesequations}.
\medskip

\noindent {\bf Outline of the article.} In Section \ref{Sec-Background}, we start by recalling and developing the controllability results obtained for the heat equations, which will be used as a building block in all our proofs. Then, in Section \ref{Sec-Strategy}, we present the strategy to prove Theorem \ref{Controle intern de NSK quantique sur le toremain} based on the analysis of the controllability of the linearized version of \eqref{eqNSKTor}: by duality, we only have to show an observability estimate on the adjoint system, which is obtained by identifying a closed subsystem of the adjoint on which the parabolic structure appears in a somewhat decoupled manner. In Sections \ref{Sec-Diag} and \ref{Sec-NonDiag} we prove the controllability of the adjoint of the resulting system depending on the diagonalizabilty of $^t\!A$ (Section \ref{Sec-Diag}) or not (Section \ref{Sec-NonDiag}). In Section \ref{Sec-Proof-Cont} we use the results of Sections \ref{Sec-Diag} and \ref{Sec-NonDiag} to show the controllability of the linearized model in suitably regular Sobolev spaces. In Section \ref{Sec-Proof-Main}, we show the local exact controllability of \eqref{eqNSKTor} using a fixed point argument. In the Appendix we  give the proof of the Carleman estimates used in this article for the complex coefficient heat equation (which coincides with the one in \cite{Localcontrollabilitytotrajectoriesfornon-homogeneousincompressibleNavier-Stokesequations} when the coefficients are real).
\begin{notation}
We set, for any $(\ell,\sigma,p)\in\Z\times\R\times([1,+\infty[\cup\{\infty\})$
$$
H^\ell(H^{\sigma}):=H^\ell(0,T;H^{\sigma}(\T_L))\ \ \ \ \text{and}\ \ \ \  L^p(H^{\sigma}):=L^p(0,T;H^{\sigma}(\T_L)),
$$
and in the same spirit
$$
\|\cdot\|_{H^\ell(H^{\sigma})}:=\|\cdot\|_{H^\ell(0,T;H^{\sigma}(\T_L))},\ \ \ \|\cdot\|_{L^p(H^{\sigma})}:=\|\cdot\|_{L^p(0,T;H^{\sigma}(\T_L))}\ \ \ \text{and} \ \ \ \|\cdot\|_{H^{\sigma}}:=\|\cdot\|_{H^{\sigma}(\T_L)}. 
$$

Throughout this article, we will also use the notation $f\lesssim g$ to express that there exists a positive constant $C$, such that $f\leq C g$. In some proofs, it will be important to underline the fact that the positive constant does not depend on some parameter. When this occurs, this will be said within the proof.
\end{notation}

\noindent\textbf{Acknowledgements.} The author expresses their gratitude to Sylvain Ervedoza for several comments on a preliminary version of this article. 

\section{Controllability of the heat equation}\label{Sec-Background}
In this section we recall and develop the control results for the complex-valued heat equation. Let $L$ be a positive real number and $\omega$ a non empty open subset of $\T_L$ such that $\overline{\omega}\subset\T_L$. Let $\zeta$ be a complex number, such that 
 \begin{equation}
     \label{2.2 BEG complex}
     \Re(\zeta)>0.
 \end{equation}
 In order to add a margin on the control zone $\omega$, we introduce a non-negative smooth cut-off function $\chi_0$ such that there exist two proper open subsets $\omega_0$ and $\omega_1$ of $\T_L$ such that 
\begin{equation}
\label{I 02/11/2023}
\omega_0\subset supp(\chi_0)\subset \omega_1\Subset \omega \ \ \text{and}\ \ \chi_0=1\ \text{on}\ \omega_0.
\end{equation}
 We consider the following controllability problem: \textit{Given $r_0$ and $f$, find a control function $v_r$ such that the solution $r$ of} 
\begin{equation}
    \label{(*)}
    \left \{
\begin{array}{lcr}
     \partial_t r - \zeta\Laplace r=f+v_{r}\chi_0\ \  \textit{in}\ (0,T)\times\T_{L},\\
     r_{|_{t=0}}=r_0\ \ \qquad\ \ \qquad \ \ \ \, \textit{in}\ \T_L,
\end{array}
\right.
\end{equation}
\textit{satisfies}
\begin{equation}
    \label{(**)}
    r_{|_{t=T}}=0\ \ \textit{in}\ \T_L.
\end{equation}
We introduce Carleman estimates derived from Carleman estimates for real coefficients heat equation established in \cite{Localcontrollabilitytotrajectoriesfornon-homogeneousincompressibleNavier-Stokesequations}. In our context we need this Carleman estimate also for complex coefficients heat equation which we establish in the appendix.

\subsection{Construction of the weight function}

Let $\psi$ be a function $\psi$ in $\mathcal{C}^2(\T_L,\R)$ such that, for every $x\in\T_L$
 \begin{equation}
     \label{2.32 EGG}
     \psi(x)\in[6,7],
 \end{equation}
and
 \begin{equation}
\label{2.35 EGG}
\inf_{\T_L\setminus\overline{\omega_0}}\lbrace |\nabla\psi|\rbrace>0.
 \end{equation}
Such a function exists according to \cite[Theorem 9.4.3, p.299]{ObservationandControlforOperatorSemigroups}.

 We choose $T_0>0$ and $\frac{1}{4}\geq T_1>0$ small enough, so that
 $$
 T_0+2T_1<T.
 $$
 For any $m\geq 2$, we introduce a weight function $\theta_m\in\mathcal{C}^2([0,T))$ such that 
 \begin{equation}
 \label{definition de theta}
     \theta_m (t)=\left \{
\begin{array}{ll}
 \displaystyle   1+\left(1-\frac{t}{T_0}\right)^{m} & \text{for all}\ t\in [0,T_0], 
 \\ 
    1 & \text{for all}\ t\in [T_0,T-2T_1],
    \\
    \theta_m\ \text{is increasing} & \text{in}\  [T-2T_1,T-T_1],
    \\
    \displaystyle\frac{1}{T-t}\  & \text{for all}\ t\in[T-T_1,T).
\end{array}
\right.
 \end{equation}
 Then we consider the following weight function, given for $s \geq 1$ and $\lambda \geq 1$, and for any $(t,x)\in [0,T)\times\T_L$ by 
 \begin{equation}
 \label{Def-Varphi}
 \vphi_{s, \lambda}(t,x):=\theta_m(t)(\lambda e^{12\lambda}-e^{\lambda\psi(x)}), 
 \quad \text{ where } 
 m=s\lambda^2e^{2\lambda},
\end{equation}
 which is always larger than $2$. 
 
 In the following, for simplifying notations, we will always denote $\varphi_{s, \lambda}$ and $\theta_m$ simply by $\varphi$ and $\theta$. 
 
 Note that $\theta$ is bounded by below by a positive constant, more precisely
 \begin{equation}
     \label{minoration de theta}
     \theta \geq 1\ \ \text{in}\ \ [0,T).
 \end{equation}
We point out that, due to the definition of $\psi$ and to Condition \eqref{2.32 EGG}, and using that $\lambda\geq 1$, we have the following bounds for any $(t,x)$ in $[0,T)\times\T_L$:
\begin{equation}
    \label{2.41 EGG}
    \frac{3}{4}\Phi(t)\leq \vphi(t,x)\leq \Phi(t),
\end{equation}
where\footnote{In \cite{LocalexactcontrollabilityforthetwoandthreedimensionalcompressibleNavierStokesequations}, the authors use $\frac{14}{15}\Phi\leq \vphi$ in $[0,T[\times \T_L$, while in this article, we only use $\frac{3}{4}\Phi\leq \vphi$ in $[0,T[\times \T_L$, which is of course weaker.}
\begin{equation}
    \Phi(t):=\theta(t)\lambda e^{12\lambda}.
\end{equation}
\subsection{Controllability results for the complex coefficients heat equation}
In this section we give some tools related to the controllability of the heat equation. We use classical method to study the controllability properties of \eqref{(*)}, which is based on the observability of the adjoint system, obtained here with the following Carleman estimates for the heat equation which we introduce in the following lemma.
\begin{lem}
\label{theorem 3.2 EGG}
Let $T>0$ and $\zeta$ a complex number satisfying \eqref{2.2 BEG complex}. There exist three positive constants $C$, $s_0\geq1$ and $\lambda_0\geq 1$, large enough, such that for any smooth function $w$ on $[0,T]\times\T_L$ and for all $s\geq s_0$, we have
\begin{align*}
s^{\frac{3}{2}}\|\theta^{\frac{3}{2}}we^{-s\vphi}\|_{L^2(L^2)}+s^{\frac 12}\|\theta^{\frac{1}{2}}\nabla & we^{-s\vphi}\|_{L^2(L^2)}  +s\|w(0)e^{-s\vphi(0)}\|_{L^2}\\
& \leq C\left(\|(-\partial_t-\overline{\zeta}\Laplace)we^{-s\vphi}\|_{L^2(L^2)}+s^{\frac 32}\|\theta^{\frac 32}\chi_0 we^{-s\vphi}\|_{L^2(L^2)}\right).
\end{align*}

\end{lem}
We give the proof of this lemma in Appendix \ref{App}. Note that, when $\zeta$ is a positive real number, this is established in \cite{Localcontrollabilitytotrajectoriesfornon-homogeneousincompressibleNavier-Stokesequations}. When $\zeta$ is a complex number satisfying \eqref{2.2 BEG complex}, the Carleman estimates in Theorem \ref{theorem 3.2 EGG} have not been done in the literature with the weight function $\varphi$ defined in \eqref{Def-Varphi}. However, similar Carleman estimates have been obtained when the weight function is the one in  \cite{ControllabilityofEvolutionequations} (see also \cite{ObservationandControlforOperatorSemigroups} Subsections 9.4, 9.5 and 9.6 for more comments on the control of parabolic equations), which is singular at $t= 0$ and at $t = T$, see in particular the works \cite{Nullcontrollabilityfortheparabolicequationwithacomplexprincipalpart} or \cite{NullControllabilityoftheComplexGinzburgLandauEquation}. 

Still, the proof of Theorem \ref{theorem 3.2 EGG} in the case of non-real parameter $\zeta$ is not completely contained in \cite{Localcontrollabilitytotrajectoriesfornon-homogeneousincompressibleNavier-Stokesequations} nor in \cite{Nullcontrollabilityfortheparabolicequationwithacomplexprincipalpart,NullControllabilityoftheComplexGinzburgLandauEquation}, and some terms need to be handled carefully, which is why we give the complete proof of Lemma \ref{theorem 3.2 EGG} in Appendix \ref{App}.

Note that we chose to take the weight function $\varphi$ in \eqref{Def-Varphi}, which is not singular at time $t = 0$, since by duality, as we will see afterwards, it allows to solve directly the control problem \eqref{(*)}--\eqref{(**)} without relying on the well-posedness of the equations. Regarding the heat equation \eqref{(*)}, this is of course not a big issue, but this is more subtle when dealing with systems. 

The proof of the following results in this subsection are left to the reader since they are straightforward adaptations of the corresponding results from \cite{LocalexactcontrollabilityforthetwoandthreedimensionalcompressibleNavierStokesequations} and \cite{LocalExactBoundaryControllabilityfortheCompressibleNavierStokesEquations}. As in \cite{LocalexactcontrollabilityforthetwoandthreedimensionalcompressibleNavierStokesequations}, these estimates lead to the following controllability result which is an adaptation of Theorem 3.3 of \cite{LocalexactcontrollabilityforthetwoandthreedimensionalcompressibleNavierStokesequations} to the case of complex coefficients heat equation.
\begin{theo}
\label{theorem 3.3 EEG}
Let $T>0$. Assume that $\zeta$ satisfies \eqref{2.2 BEG complex}. There exist constants $C>0$ and $s_0\geq 1$ such that for all $s\geq s_0$, for all $f\in L^{2}(0,T;L^2(\T_L))$ satisfying
\begin{equation}
    \label{3.9 EGG}
    \|\theta^{-\frac{3}{2}}f e^{s\vphi}\|_{L^2(L^2)}<+\infty
\end{equation}
 and $r_0\in L^2(\T_L)$, there exists a solution $(r,v_{r})$ of the control problem \eqref{(*)}-\eqref{(**)} which furthermore satisfies the following estimate:
 \begin{align}
     s^{\frac{3}{2}}\|re^{s\vphi}\|_{L^2(L^2)}+\|\theta^{-\frac 32}\chi_0v_{r}e^{s\vphi}\|_{L^2(L^2)}+s^{\frac 12}\|\theta^{-1}\nabla &re^{s\vphi}\|_{L^2(L^2)}\nonumber\\
     &\leq C\left(\|\theta^{-\frac 32}fe^{s\vphi}\|_{L^2(L^2)}+s^{\frac 12}\|r_0e^{s\vphi(0)}\|_{L^2}\right).\label{3.10 EGG}
 \end{align}
 Moreover, the solution $(r,v_r)$ can be obtained through a linear operator in $(r_0,f)$.
\end{theo}
We need also to know what can be done when the source term $f$ is more regular and lies in $L^2(0,T;H^{1}(\T_L))$ or in $L^2(0,T; H^{2}(\T_L))$ (\textit{c.f.} \cite{Localcontrollabilitytotrajectoriesfornon-homogeneousincompressibleNavier-Stokesequations} Proposition 3.4). By adapting the proof of Proposition 3.4 of \cite{LocalexactcontrollabilityforthetwoandthreedimensionalcompressibleNavierStokesequations}, we deduce the following lemma.
\begin{lem}
\label{proposition 3.4 EGG}
Let $T>0$. Consider the solution $(r,v_r)$ construct in Theorem \ref{theorem 3.3 EEG}. Then, with the above notation, for some constant $C>0$ independent of $s$, we have the following properties:
\begin{enumerate}
    \item ($H^2$ \text{regularity of the control} $v_r\in L^{2}(0,T;H^2(\T_L))$ and
    $$
    \|\chi_0v_re^{3s\Phi/4}\|_{L^2(H^2)}\leq C\left(\|\theta^{-\frac 32}f e^{s\vphi}\|_{L^2(L^2)}+\|r_0e^{s\Phi(0)}\|_{L^2}\right).
    $$
    \item ($H^3$ \textit{regularity estimate for the state}) if $r_0\in H^2(\T_L)$, $fe^{3s\Phi/4}\in L^2(0,T; H^1(\T_L))$ and  $\theta^{-\frac{3}{2}}fe^{s\vphi}\in L^2(0,T; L^2(\T_L))$, then $r\in L^2(0,T;H^3(\T_L))$ and
    $$
    \|re^{3s\Phi/4}\|_{L^2(H^3)}\leq C\left(\|fe^{3s\Phi/4}\|_{L^2(H^1)}+\|\theta^{-\frac 32}fe^{s\vphi}\|_{L^2(L^2)}+\|r_0e^{s\Phi(0)}\|_{H^2}\right).
    $$
\end{enumerate}
\end{lem}
For later use, we also present the following results, which contain a shift in the weight $\theta$ compared to Theorem \ref{theorem 3.3 EEG} and Lemma \ref{proposition 3.4 EGG}. The following results are modified versions of Theorem 3.3 of \cite{LocalExactBoundaryControllabilityfortheCompressibleNavierStokesEquations} which also can be adapted to the case of complex coefficients heat equation.
\begin{theo}
\label{theorem 3.3 M}
Let $T>0$. Assume that $\zeta$ satisfies \eqref{2.2 BEG complex}. There exist a positive constant $C$ and a real number $s_0\geq 1$ such that for all $s\geq s_0$ and for all $f$ satisfying
\begin{equation}
\label{3.9 EGG version Jordan}
\|\theta^{-1}f e^{s\vphi}\|_{L^2(0,T;L^2(\T_L))}<+\infty
\end{equation}
and $r_0\in H^1(\T_L)$, the solution $(r,v_r)$ of the control problem \eqref{(*)}-\eqref{(**)} satisfies
\begin{align}
s^{\frac{3}{2}}\|\theta^{\frac 12}re^{s\vphi}&\|_{L^2(L^2)}+\|\theta^{-1}\chi_0 v_{r} e^{s\vphi}\|_{L^2(L^2)}+s^{\frac 12}\|\theta^{-\frac 12}\nabla re^{s\vphi}\|_{L^2(L^2)}+s^{-\frac 12}\|\theta^{-\frac 32}\nabla^2 re^{s\vphi}\|_{L^2(L^2)}\nonumber\\
&\leq C\left(\|\theta^{-1}f e^{s\vphi}\|_{L^2(L^2)}+s^{\frac 12}\|r_0e^{s\vphi(0)}\|_{L^2}+s^{-\frac{1}{2}}\|\nabla r_0e^{s\vphi(0)}\|_{L^2}\right)\label{46 M}.
\end{align}
Moreover, the solution $(r,v_r)$ can be obtained through a linear operator in $(r_0,f)$.
\end{theo}
\begin{lem}
\label{proposition 3.4 EGG version Jordan}
Let $T>0$. Consider the solution $(r,v_r)$ constructed in Theorem \ref{theorem 3.3 M}. Then, with the above notations, for some constant $C>0$ independent of $s$, we have the following properties:
\begin{enumerate}
    \item ($H^2$ \text{regularity of the control}) $v_r\in L^{2}(0,T;H^2(\T_L))$ and
    $$
    \|\chi_0v_re^{3s\Phi/4}\|_{L^2(H^2)}\leq C\left(\|\theta^{-1}f e^{s\vphi}\|_{L^2(L^2)}+\|r_0e^{s\Phi(0)}\|_{L^2}\right).
    $$
    \item ($H^3$ \textit{regularity estimate for the state}) if $r_0\in H^2(\T_L)$, $fe^{3s\Phi/4}\in L^2(0,T; H^1(\T_L))$ and  $\theta^{-1}fe^{s\vphi}\in L^2(0,T; L^2(\T_L))$, then $r\in L^2(0,T;H^3(\T_L))$ and
    $$
    \|re^{3s\Phi/4}\|_{L^2(H^3)}\leq C\left(\|fe^{3s\Phi/4}\|_{L^2(H^1)}+\|\theta^{-1}fe^{s\vphi}\|_{L^2(L^2)}+\|r_0e^{s\Phi(0)}\|_{H^2}\right).
    $$
\end{enumerate}
\end{lem}


\section{Strategy}\label{Sec-Strategy} In order to perform a perturbative argument by fixed-point, we first recast the system into a more user-friendly shape. 
Let us set
\begin{equation}
  \label{definition of a}  
a:=\frac{\rho}{\rho_{\star}}-1.
\end{equation}
We are then led to study the following system
\begin{equation}
\label{NSK quantique}
\left \{
\begin{array}{lcr}

     \partial_t a + \divergence(u)=f_{a}(a,u)+v_a\mathbf{1}_{\omega}\ \ \ \ \ \ \ \ \ \ \ \ \ \ \ \ \ \ \ \ \ \ \ \ \ \ \ \ \ \ \ \ \ \ \ \ \ \ \ \qquad\qquad\ \ \ \text{in}\ \ (0,T)\times\T_L,\\ 
      \partial_t u-\mu_{\star} \Laplace{u}-(\mu_{\star}+\nu_{\star})\nabla\divergence (u) +p_{\star}\nabla a-\kappa_{\star}\nabla\Laplace{a}=f_{u}(a,u)+v_u\mathbf{1}_{\omega}\ \ \text{in}\ \ (0,T)\times\T_L,
     
\end{array}
\right.
\end{equation}
where 
\begin{equation*}
\kappa_{\star}:=\rho_{\star}\kappa(\rho_{\star}),
\quad 
\mu_{\star}:=\rho_{\star}^{-1}\mu(\rho_{\star}),
\quad
 \nu_{\star}:=\rho_{\star}^{-1}\nu(\rho_{\star}),
\quad p_{\star}:=P^{'}(\rho_{\star})
\end{equation*}
and
\begin{equation}
\label{definition non linearite}
\left \{
\begin{array}{lcr}
  f_{a}(a,u) :=-u\cdot \nabla a, \\
  f_{u}(a,u) :=f_{u}^{1}(a,u)+f_{u}^{2}(a,u)+f_{u}^{3}(a,u)+f_{u}^{4}(a)+f_{u}^{5}(a),
\end{array}
\right.
\end{equation}
with
\begin{equation*}
\begin{cases}
f_{u}^{1}(a,u):=-(a+1)u\cdot\nabla u,\\
\displaystyle f^{2}_{u}(a,u):=[\divergence\left(2\underline{\mu}(a)D_S(u))+\nabla(\underline{\nu}(a)\divergence{u})\right)],\\
\displaystyle f_{u}^{3}(a,u):=(\partial_tu) a,\\
\displaystyle f_{u}^{4}(a):= \underline{P}'(a)\nabla a,\\
\displaystyle f_{u}^{5}(a):=(a+1)\nabla\Big(\underline{\kappa}(a) \Laplace a+\nabla\underline{\kappa}(a)\cdot\nabla a\Big)
\end{cases}
\end{equation*}
and
\begin{equation*}
\begin{cases}
\underline{\kappa}(a):=\rho_{\star}(\kappa(\rho_{\star}a+\rho_{\star})-\kappa(\rho_{\star})),\\
\displaystyle \underline{\mu}(a):=\rho_{\star}^{-1}(\mu(\rho_{\star}a+\rho_{\star})-\mu(\rho_{\star})),\\
\displaystyle \underline{\nu}(a):=\rho_{\star}^{-1}(\nu(\rho_{\star}a+\rho_{\star})-\nu(\rho_{\star})),\\
\displaystyle \displaystyle \underline{\nu}(a):=P^{'}(\rho_{\star}a+\rho_{\star})-P^{'}(\rho_{\star}).
\end{cases}
\end{equation*}
This form of System \eqref{eqNSKTor} is more convenient since the linearized equations around $(\rho_{\star},0)$ explicitly appear in the left-hand side of \eqref{NSK quantique}. 

Also note that we can recover strong solutions of \eqref{eqNSK} from strong solutions of \eqref{NSK quantique} by using \eqref{definition of a} under the form
$$
\rho=\rho_{\star}a+\rho_{\star}.
$$
This leads to studying the distributed null controllability problem associated to \eqref{NSK quantique} on $\T_L$. Then we consider the following control problem: \textit{Given $(a_0,u_0)$ small enough, find control functions $(v_a,v_u)$ on $(0,T)\times\T_L$ such that the solution $(a,u)$ of} 
\begin{equation}
\label{NSK quantique internal control}
\left \{
\begin{array}{lcr}

     \partial_t a + \divergence(u)=f_a(a,u)+v_a\mathbf{1}_{\omega}\ \ \, \, \, \ \ \ \ \ \ \qquad\qquad\qquad\qquad\qquad\qquad\qquad\, \, \ \textit{in}\ (0,T)\times\T_{L}, \\ 
      \partial_t u-\mu_{\star} \Laplace{u}-(\mu_{\star}+\nu_{\star})\nabla\divergence (u) +p_{\star}\nabla a-\kappa_{\star}\nabla\Laplace{a}=f_u(a,u)+v_u\mathbf{1}_{\omega}\ \,\textit{in}\ (0,T)\times\T_{L},
     
\end{array}
\right.
\end{equation}
satisfies 
\begin{equation}
    \label{condition final}
    (a,u)_{|_{t=T}}=(0,0)\ \ \ \textit{in}\ \T_{L}.
\end{equation}
Then, Theorem \ref{main resultmain} is equivalent to the following theorem.\\
\begin{theo}
\label{Controle intern de NSK quantique sur le tore}
Let $d\in\{1,2,3\}$ and $T>0$. There exists $\delta>0$ such that, for all $(a_0,u_0)\in H^2(\T_L)\times H^1(\T_L)$ satisfying 
\begin{equation}
\label{condition de petitesse pour donne initial tore}
    \|(a_0,u_0)\|_{H^2\times H^1}\leq \delta,
\end{equation}
there exist control functions $(v_a,v_u)$ in $L^2(0,T;H^2(\T_L))\times L^2(0,T;H^1(\T_L))$ and a corresponding controlled trajectory $(a,u)$ solving \eqref{NSK quantique internal control} with initial data $(a_0,u_0)$ and satisfying the control requirement \eqref{condition final}. 
Besides, the controlled trajectory $(a,u)$ enjoys the following regularity 
\begin{align*}
	& a \in \mathcal{C}([0,T]; H^2(\T_L))\cap L^2(0,T;H^3(\T_L))\cap H^1(0,T;H^1(\T_L)),
	\\
	& u \in \mathcal{C}([0,T]; H^1(\T_L))\cap L^2(0,T;H^2(\T_L))\cap H^1(0,T;L^2(\T_L)).
\end{align*}
and the following bound from below
$$
	\inf_{(t,x) \in [0,T] \times \T_L} a(t,x) > -1.
$$
\end{theo}
To take into account the support of the control functions $(v_a,v_u)\mathbf{1}_{\omega}$ and facilitate regularity issues, we replace $\mathbf{1}_{\omega}$ by a smooth cut-off function $\chi\in\mathcal{C}^{\infty}_{c}(\omega, [0,1])$ satisfying
\begin{equation*}
\begin{array}{lcr}
  \chi= 1\ \text{on}\ \omega_1,
\end{array}
\end{equation*}
recall \eqref{I 02/11/2023} for the definition of $\omega_1$ compared to $\omega$. In other words, we consider the following control problem: \textit{Given $(a_0, u_0)$ small in $H^2(\T_L)\times
H^1(\T_L)$, find control functions $(v_a,v_u)$ in $L^2(H^1)\times L^2(L^2)$ such that the solution $(a,u)\in L^2(H^3)\times L^2(H^2)$ of}
\begin{equation}
\label{NSK quantique control}
\left \{
\begin{array}{lcr}
     \partial_t a + \divergence(u)=f_a(a,u)+v_a\chi\ \ \qquad\qquad\qquad\qquad\qquad\qquad\qquad\ \ \ \ \ \ \, \, \, \, \, \textit{in}\ (0,T)\times\T_{L}, \\ 
      \partial_t u- \mu_{\star} \Laplace{u}-(\mu_{\star}+\nu_{\star})\nabla\divergence (u) +p_{\star}\nabla a-\kappa_{\star}\nabla\Laplace{a}=f_u(a,u)+v_u\chi\ \textit{in}\ (0,T)\times\T_{L},      
\end{array}
\right.
\end{equation}
\textit{with initial data}
\begin{equation}
    \label{condition initial}
    (a,u)_{|_{t=0}}=(a_0,u_0)\ \ \textit{in}\ \T_{L},
\end{equation} 
\textit{satisfies} \eqref{condition final}.

Of course, this will be done by the analysis of the control properties of the linearized system
\begin{equation}
\label{NSK quantique Sl}
\left \{
\begin{array}{lcr}
     \partial_t a + \divergence(u)=f_a+v_a\chi\ \qquad\qquad\qquad\qquad\qquad\qquad\qquad\ \ \ \ \ \ \ \, \ \, \, \ \ \text{in}\ (0,T)\times\T_{L}, \\ 
      \partial_t u-\mu_{\star} \Laplace{u}-(\mu_{\star}+\nu_{\star})\nabla\divergence (u) +p_{\star}\nabla a-\kappa_{\star}\nabla\Laplace{a}=f_u+v_u\chi\ \ \, \text{in}\ (0,T)\times\T_{L},     
\end{array}
\right.
\end{equation}
where $f_a$ and $f_u$ are given\footnote{Note that index $a$ and $u$, aims to indicate in which equations $f_a$ and $f_u$ appear in the lines of the system according to the terms $\partial_ta $ and $\partial_t u$ respectively. In particular, $f_a$ and $f_u$ do not depend on $a$ and $u$ except if it explicitly appear, as in $f_a(a,u)$ or $f_u(a,u)$. In this article, we use similar notations for source terms in \eqref{NSK quantique Sl}, \eqref{NSK quantique Sl*}, \eqref{NSK quantique Sloxc*}, \eqref{NSK quantique Slcc*}, \eqref{NSK quantique Slcc**}, \eqref{NSK quantique Slcc* Jordan} and \eqref{NSK quantique Slcc** Jordan}.}. Namely, we establish the following theorem whose proof is given in Section 6.
\begin{theo}
\label{theorem 4.1 EGG}
Let $T>0$. There exist a real number $s_0\geq 1$ and a positive constant $C$, such that for any $(a_0,u_0)\in H^{2}(\T_L)\times H^1(\T_L)$ and $f_a,f_u$ such that $f_ae^{\frac{4s_0\Phi}{3}}\in L^2(0,T;H^1(\T_L))$ and $f_ue^{\frac{4s_0\Phi}{3}}\in L^2(0,T;L^2(\T_L))$, there exist two control functions $v_a$ and $v_u$ and a corresponding controlled trajectory $(a,u)$ solving \eqref{NSK quantique Sl} with initial data $(a_0,u_0)$, satisfying the controllability requirements \eqref{condition final} and depending linearly on the data $(a_0,u_0,f_a,f_u)$. Besides, we have the following estimate
\begin{align}
\|(\partial_ta,\partial_tu)&e^{\frac{2s_0\Phi}{3}}\|_{L^2(H^1)\times L^2(L^2)}+\|(a,u)e^{\frac{2s_0\Phi}{3}}\|_{L^2(H^3)\cap L^{\infty}(H^2)\times L^2(H^2)\cap L^{\infty}(H^1)}\nonumber\\
&+\|\chi  (v_a,v_u)  e^{\frac{3s_0\Phi}{4}}\|_{L^2(H^1)\times L^2(L^2)} \leq C\left(\|(f_a,f_u)e^{\frac{4s_0\Phi}{3}})\|_{L^2(H^1)\times L^2(L^2)}+\|(a_0,u_0)\|_{H^2\times H^1}\right)\label{4.2 EGG}.
\end{align}
This allows us to define a linear operator $\mathcal{G}$ defined on the space
$$
\enstq{(a_0,u_0,f_a,f_u)\in H^2\times H^1\times L^2(H^1)\times L^2(L^2)}{f_ae^{\frac{4s_0\Phi}{3}}\in L^2(H^1)\ \text{and}\ f_ue^{\frac{4s_0\Phi}{3}}\in L^2(L^2) },
$$
by 
$$
\mathcal{G}(a_0,u_0,f_a,f_u)=(a,u),
$$
where $(a,u)$ is the controlled trajectory, with initial condition $(a_0,u_0)$ and forces $(f_a,f_u)$, satisfying the control requirement \eqref{condition final} and Estimate \eqref{4.2 EGG}.
\end{theo}
Since System  \eqref{NSK quantique Sl} is linear, its controllability in $L^{2}(H^3)\times L^{2}(H^2)$ is equivalent to the following observability estimate
\begin{align*}
    \|\sigma e^{-s_0\Phi}\|_{L^2(H^{-1})}+\|&\sigma(0)e^{-s_0\Phi(0)}\|_{H^{-2}}+\|ze^{-\frac{4s_0\Phi}{3}}\|_{L^2(L^2)}+\|z(0)e^{-\frac{4s_0\Phi(0)}{3}}\|_{H^{-1}}\\
    &\lesssim \|(g_{\sigma},g_z)e^{-\frac{3s_0\Phi}{4}}\|_{L^2(H^{-3})\times L^2(H^{-2})}+\|\chi (\sigma,z)e^{-\frac{3s_0\Phi}{4}}\|_{L^2(H^{-1})\times L^2(L^2)},
\end{align*}
where $(\sigma,z)$ is a solution of the following adjoint system
\begin{equation}
\label{NSK quantique Sl*}
\left \{
\begin{array}{lcr}
     -\partial_t \sigma - p_{\star}\divergence(z)+\kappa_{\star}\Laplace\divergence(z)=g_{\sigma}\ \qquad\ \ \ \ \, \, \, \ \text{in}\ (0,T)\times\T_{L}, \\ 
      -\partial_t z- \nabla \sigma-\mu_{\star} \Laplace{z}-(\mu_{\star}+\nu_{\star})\nabla\divergence (z) =g_z\ \ \text{in}\ (0,T)\times\T_{L},
     
\end{array}
\right.
\end{equation}
with $(g_{\sigma},g_{z})\in L^2(H^{-1})\times L^2(L^2)$. 

The main idea to get this observability estimate for \eqref{NSK quantique Sl*} is based on the fact that, with 
$$
q:=\divergence(z),
$$
$(\sigma, q)$ satisfies the closed subsystem 
\begin{equation}
\label{NSK quantique Sloxc*}
\left \{
\begin{array}{lcr}

     -\partial_t \sigma - p_{\star} q+\kappa_{\star}\Laplace q=g_{\sigma}\ \ \qquad\ \ \, \ \text{in}\ (0,T)\times\T_{L}, \\ 
      -\partial_t q-\Laplace\sigma-(2\mu_{\star}+\nu_{\star})\Laplace q =g_q\ \, \text{in}\ (0,T)\times\T_{L},
     
\end{array}
\right.
\end{equation}
where $g_q :=\divergence(g_z)$.

We will thus base our analysis on the following observability estimate
\begin{align*}
    \|(\sigma,q)e^{-s_0\Phi}&\|_{L^2(H^{-1})\times L^2(H^{-1})}+\|(\sigma(0),q(0)) e^{-s_0\Phi(0)}\|_{H^{-2}\times H^{-2}} \nonumber\\
    & \lesssim \|(g_{\sigma},g_{z})e^{-\frac{3s_0\Phi}{4}}\|_{L^2(H^{-3})\times L^2(H^{-2})} +\|\chi_0(\sigma,q)e^{-\frac{3s_0\Phi}{4}}\|_{L^2(H^{-1})\times L^2(H^{-1})},\label{changement inconnu dans inegalite de dualite}
\end{align*}
for the solutions $(\sigma, q)$ of \eqref{NSK quantique Sloxc*}, where $\chi_0$ is the cut-off function in \eqref{I 02/11/2023}.

In order to do that and to underline the parabolic behavior of \eqref{NSK quantique Sloxc*}, we rely on the analysis of the matrix 
\begin{equation}
\label{La fameuse matrice}
\begin{pmatrix}
 0 & \kappa_{\star}  \\
-1 & -(2\mu_{\star}+\nu_{\star})
\end{pmatrix},
\end{equation}
which coincides with the matrix $^t\!A$, where $A$ is given in \eqref{matrice A}.

As said in the introduction, our analysis will then be divided into two cases: when $^t\!A$ is diagonalizable (equivalently $A$), and when $^t\!A$ (equivalently $A$) is not diagonalizable.

When $A$ is diagonalizable, System \eqref{NSK quantique Sloxc*} is a parabolic or parabolic/dispersive system in which the coupling is done through lower order terms. We can then use directly Lemma \ref{theorem 3.2 EGG} to obtain $L^2$ observability results for System \eqref{NSK quantique Sloxc*}. However, there still remains an additional difficulty to obtain observability results for System \eqref{NSK quantique Sloxc*} in negative index Sobolev spaces. This is done by duality by obtaining controllability results for the adjoint of System \eqref{NSK quantique Sloxc*} in spaces of higher regularity, based on Theorem \ref{theorem 3.3 EEG} and Lemma \ref{proposition 3.4 EGG}.

When $A$ is not diagonalizable, System \eqref{NSK quantique Sloxc*} is a parabolic system in which the coupling is done through the leading order. This does not prevent us to follow the same strategy, but one needs to perform a slight shift of a power of the function $\theta$ in the control results. This is why we introduced Theorem \ref{theorem 3.3 M} and Lemma \ref{proposition 3.4 EGG version Jordan}.

\subsection{Diagonalizable case}
We assume in this subsection that \eqref{La fameuse matrice} is diagonalizable. 
Then, setting 
\begin{equation*}
     \zeta_{+}:=\frac{(2\mu_{\star}+\nu_{\star})-D}{2}, \quad
      \zeta_{-} :=\frac{(2\mu_{\star}+\nu_{\star})+D}{2},
      \quad
      \text{ where \footnote{We use the convention $\sqrt{x}:=i\sqrt{-x}$ if $x\leq 0$.} } 
      D:=\sqrt{(2\mu_{\star}+\nu_{\star})^2-4\kappa_{\star}}, 
\end{equation*}
(note that $\Re(\zeta_{\pm})>0$ according to {\bf (H1)}), the matrix $^t\! A$ is equivalent to the diagonal matrix $\text{diag}(\zeta_+, \zeta_-)$. This can be done through the $2\times 2$ invertible matrix $Q$ given by 
\begin{equation}
	\label{Def-Q}
Q:=
\begin{pmatrix}
\displaystyle \frac{\zeta_{+}}{D} &\displaystyle  -\frac{\kappa_{\star}}{D} \smallskip\\
\displaystyle \frac{-\zeta_{+}}{D} &\displaystyle  \frac{\kappa_{\star}}{D} \\
\end{pmatrix}.
\end{equation}
In particular, $(\sigma, q)$ solves System \eqref{NSK quantique Sloxc*} if and only if the new unknowns 
$$
\begin{pmatrix}
y^{+} \\
y^{-} 
\end{pmatrix}:=Q
\begin{pmatrix}
\sigma
\\ 
q
\end{pmatrix},
$$
satisfy
\begin{equation}
\label{NSK quantique Slcc*}
\left \{
\begin{array}{lcr}
     -\partial_t y^{+} - \zeta_{+}\Laplace y^{+}=g_{y^{+}}+\alpha_1y^{+}+\alpha_2y^{-}\ \ \ \, \text{in}\ (0,T)\times\T_{L}, \\ 
      -\partial_t y^{-}-\zeta_{-}\Laplace y^{-} =g_{y^{-}}+\alpha_3y^{+}+\alpha_4y^{-}\ \ \ \,\text{in}\ (0,T)\times\T_{L},
      
\end{array}
\right.
\end{equation}
with
$$
\alpha_1:=\frac{\zeta_{+}\zeta_{-}p_{\star}}{(\zeta_{+}+\zeta_{-})\kappa_{\star}}, 
\quad 
\alpha_2:=\frac{\zeta_{-}^{2}p_{\star}}{(\zeta_{+}+\zeta_{-})\kappa_{\star}}, 
\quad
\alpha_3:=\frac{\zeta_{+}p_{\star}}{\zeta_{+}+\zeta_{-}}, 
\quad 
\alpha_4:=-\frac{\zeta_{+}p_{\star}}{\zeta_{+}+\zeta_{-}},
$$
 and 
$$
\begin{pmatrix}
g_{y^{+}} \\
g_{y^{-}} 
\end{pmatrix}:=Q
\begin{pmatrix}
g_{\sigma} \\
g_q 
\end{pmatrix}.
$$
The $L^2$-observability for this system is well-known and follows directly from the Carleman estimates for parabolic system (see \cite{ControllabilityofEvolutionequations} and \cite{ObservationandControlforOperatorSemigroups} Subsections 9.4, 9.5 and 9.6 for more comments). A part of the difficulty here is to obtain this observability estimate in negative index Sobolev spaces. To this aim, we will adapt the strategy from \cite{LocalexactcontrollabilityforthetwoandthreedimensionalcompressibleNavierStokesequations}, and prove the following controllability result in Sobolev spaces of strong regularity: 
 \\
 \textit{Given $(r_{0}^{+},r_{0}^{-})$ in $H^2(\T_L)\times H^2(\T_L)$, find two control $(v_{r^{+}},v_{r^{-}})$ in $L^2(H^1)\times L^2(H^1)$ such that the solution $(r^{+},r^{-})$ of} 
\begin{equation}
\label{NSK quantique Slcc**}
\left \{
\begin{array}{lcr}
     \partial_t r^{+} - \overline{\zeta}_{+}\Laplace r=f_{r^{+}}+\overline{\alpha}_{1}r^{+}+\overline{\alpha}_{3}r^{-}+\chi_0 v_{r^{+}}\ \ \ \textit{in}\ (0,T)\times\T_{L}, \\ 
      \partial_t r^{-}-\overline{\zeta}_{-}\Laplace r^{-} =f_{r^{-}}+\overline{\alpha}_{2}r^{+}+\overline{\alpha}_{4}r^{-}+\chi_0 v_{r^{-}}\ \textit{in}\ (0,T)\times\T_{L},
      \end{array}
\right.
\end{equation}
\textit{satisfies} 
\begin{equation}
\label{NSK quantique Slcc** initial et final}
(r^{+},r^{-})_{|_{t=0}}=(r^{+}_{0},r^{-}_{0})\ \ \ \textit{and}\ \ \ (r^{+},r^{-})_{|_{t=T}}=(0,0)\ \ \textit{in}\ \ \T_L,
\end{equation}
\textit{and belongs to} $L^2(H^3)\times L^2(H^3)$. 
We treat this problem in Section 4. The proof is based on the estimates of Theorem \ref{theorem 3.3 EEG} and Lemma \ref{proposition 3.4 EGG}.

\subsection{Non-diagonalizable case}

In this subsection we assume that \eqref{La fameuse matrice} is non-diagonalizable.\\ 
In this case we set 
$$
\zeta=\frac{2\mu_{\star}+\nu_{\star}}{2}>0,
$$ 
and, to make $^t\! A$ triangular, we consider  the $2\times2$ invertible matrix $R$ given by 
\begin{equation}
\label{Def-R}
R:=
\begin{pmatrix}
1 & 0 \smallskip\\
\displaystyle \frac{1}{\zeta} & 1 
\end{pmatrix}.
\end{equation}
Through explicit computations, we then check that  $(\sigma, q)$ solves System \eqref{NSK quantique Sloxc*} if and only if the new unknowns 
$$
\begin{pmatrix}
y^{+} \\
y^{-} 
\end{pmatrix}:=R
\begin{pmatrix}
\sigma \\
q 
\end{pmatrix}
$$
satisfies
\begin{equation}
\label{NSK quantique Slcc* Jordan}
\left \{
\begin{array}{lcr}
     -\partial_t y^{+} - \zeta\Laplace y^{+}=g_{y^{+}}+\beta_1 y^{+}+\beta_2 y^{-}-\kappa_{\star}\Laplace y^{-}\  \ \text{in}\ (0,T)\times\T_{L}, \\ 
      -\partial_t y^{-}-\zeta\Laplace y^{-} =g_{y^{-}}+\beta_3 y^{+}+\beta_4 y^{-}\ \ \qquad\ \ \, \ \ \ \ \ \text{in}\ (0,T)\times\T_{L},
\end{array}
\right.
\end{equation}
with 
$$
	\beta_1:=-\frac{p_{\star}}{\zeta},
	\quad
	 \beta_2:=p_{\star}, 
	 \quad
		\beta_3:=-\frac{p_{\star}}{\zeta}, 
	\quad 
	\beta_4:=\frac{p_{\star}}{\zeta}.
$$
and 
$$
\begin{pmatrix}
g_{y^{+}} \\
g_{y^{-}} 
\end{pmatrix}:=R
\begin{pmatrix}
g_{\sigma} \\
g_q 
\end{pmatrix}.
$$
Since System \eqref{NSK quantique Slcc* Jordan} is linear, its observability is equivalent to the controllability statement for this adjoint system written in the dual variables $(r^{+},r^{-})$, where the adjoint is taken with respect to the variables $(y^{+},y^{-})$. And again, to get an observability result in Sobolev spaces of weak regularity on $(y^+,y^-)$, we will consider the controllability of the adjoint in Sobolev spaces of high regularity as follows.
\\
\textit{Given $(r^{+}_{0},r^{-}_{0})$ in $H^2(\T_L)\times H^2(\T_L)$, find two control $(v_{r^{+}},v_{r^{-}})$ in $L^2(H^1)\times L^2(H^1)$ such that the solution $(r^{+},r^{-})$ of the}
\begin{equation}
\label{NSK quantique Slcc** Jordan}
\left \{
\begin{array}{lcr}
     \partial_t r^{+} - \zeta\Laplace r^{+}=f_{r^{+}}+\beta_1r^{+}+\beta_3r^{-}+\chi_0 v_{r^{+}}\ \ \ \ \ \ \ \ \ \ \qquad\textit{in}\ (0,T)\times\T_{L}, \\ 
     \partial_t r^{-}-\zeta\Laplace r^{-} =f_{r^{-}}+\beta_2r^{+}+\beta_4 r^{-}-\kappa_{\star}\Laplace r^{+}+\chi_0 v_{r^{-}}\ \, \ \textit{in}\ (0,T)\times\T_{L},
      
\end{array}
\right.
\end{equation}
\textit{satisfies} 
$$
(r^{+},r^{-})_{|_{t=0}}=(r^{+}_{0},r^{-}_{0})\ \ \ \textit{and}\ \ \ (r^{+},r^{-})_{|_{t=T}}=(0,0) \ \ \textit{in}\ \ \T_L,
$$
\textit{and belongs to} $L^2(H^3)\times L^2(H^3)$.

We treat this problem in Section 5. The proof is based on Theorem \ref{theorem 3.3 M} and Lemma \ref{proposition 3.4 EGG version Jordan}.
%
%
\section{Controllability of (\ref{NSK quantique Slcc**}): The diagonalizable case.}\label{Sec-Diag}
This section is devoted to the controllability of \eqref{NSK quantique Slcc**}. We recall that this system corresponds to the case in which \eqref{La fameuse matrice} is diagonalizable.  We aim to establish the following theorem.
 \begin{lem}
 \label{control Slcc**}
Let $T>0$. Let $(r^{+}_{0},r^{-}_{0})\in H^2(\T_L)\times H^2(\T_L)$. There exist a positive constant $C$ and a real number $s_0\geq 1$ such that for all $s\geq s_0$, for all $f_{r^{+}} $ and $f_{r^{-}}$ in $L^2(L^2)$ such that 
\begin{equation}
\label{3.1 EGG}
\|\theta^{-\frac 32}(f_{r^{+}},f_{r^{-}})e^{s\vphi}\|_{L^2(L^2)}<+\infty
\end{equation}
and
\begin{equation}
\label{3.3 EGG}
  (f_{r^{+}},f_{r^{-}})e^{s\Phi}\in L^2(0,T;H^1(\T_L)),  
\end{equation}
there exists a controlled trajectory $(r^{+},r^{-})$ solving \eqref{NSK quantique Slcc**} and satisfying the following estimate
\begin{align}
    \|(r^{+},r^{-}) e^{3s\Phi/4}&\|_{L^2(H^3)}+\|\chi_0(v_{r^{+}},v_{r^{-}})e^{3s\Phi/4}\|_{L^2(H^1)}\nonumber\\ 
    & \leq C \left(\|(f_{r^{+}},f_{r^{-}})e^{s\Phi}\|_{L^2(H^1)}+\|(r^{+}_{0},r^{-}_{0})e^{s\Phi(0)}\|_{H^2},\label{3.4 EGG}\right).
\end{align}
\end{lem}
\begin{proof}
We will prove the controllability of \eqref{NSK quantique Slcc**} by a fixed-point argument based on the control results obtained in Theorem \ref{theorem 3.3 EEG} and Lemma \ref{proposition 3.4 EGG}. The proof follows the strategy of \cite{LocalexactcontrollabilityforthetwoandthreedimensionalcompressibleNavierStokesequations}.\\
\textit{Existence of the solution to the control problem.} We construct the controlled trajectory using a Banach fixed-point argument. We introduce the linear space
$$
\mathcal{C}_s:=\enstq{r\in L^2(0,T; L^2(\T_L))}{re^{s\vphi}\in L^2(0,T;L^2(\T_L)) }.
$$
For $\Tilde{r}^{+}$ and $\Tilde{r}^{-}$ in $\mathcal{C}_s$, we introduce 
\begin{equation*}
    \left \{
\begin{array}{lcr}
     \Tilde{f}_{r^{+}}:=\Tilde{f}_{r^{+}}(\Tilde{r}^{+},\Tilde{r}^{-})=f_{r^{+}}+\overline{\alpha}_1\Tilde{r}^{+}+\overline{\alpha}_3\Tilde{r}^{-},\\
     
    \Tilde{f}_{r^{-}}:=\Tilde{f}_{r^{-}}(\Tilde{r}^{+},\Tilde{r}^{-})=f_{r^{-}}+\overline{\alpha}_2\Tilde{r}^{+}+\overline{\alpha}_4\Tilde{r}^{-}.
    
\end{array}
\right.
\end{equation*}
As $f_{r^{+}}$ and $f_{r^{-}}$ satisfy \eqref{3.1 EGG}, for every $(\Tilde{r}^{+},\Tilde{r}^{-})$ in $\mathcal{C}_s\times\mathcal{C}_s$, $\Tilde{f}_{r^{+}}$ and $\Tilde{f}_{r^{-}}$ satisfy Assumption \eqref{3.9 EGG} of Theorem \ref{theorem 3.3 EEG}. In fact, Theorem \ref{theorem 3.3 EEG} provides two linear maps $(r^{+}_{0},\Tilde{f}_{r^{+}})\mapsto (r^{+},v_{r^{+}})$ and $(r^{-}_{0},\Tilde{f}_{r^{-}})\mapsto (r^{-},v_{r^{-}})$. Therefore, one can define a map $\Lambda_s$ on $\mathcal{C}_s\times\mathcal{C}_s$ which to a data $(\Tilde{r}^{+},\Tilde{r}^{-})$ in $\mathcal{C}_s\times\mathcal{C}_s$ associates $(r^{+},r^{-})$ where $r^{+}$ and $r^{-}$ are respectively the solutions of the following controlled systems
\begin{equation*}
    \left \{
    \begin{array}{lcr}
    \partial_t r^{+}-\overline\zeta_{+}\Laplace r^{+}=\Tilde{f}_{r^{+}}+v_{r^{+}}\chi_0\ \text{in}\ (0,T)\times\T_L,\\
    r^{+}_{|_{t=0}}=r^{+}_{0},\ \ r^{+}_{|_{t=T}}=0\ \ \ \ \ \ \, \qquad \text{in}\ \T_L
    \end{array}
    \right.
\end{equation*}
 and
\begin{equation*}
    \left \{
    \begin{array}{lcr}
    \partial_t r^{-}-\overline\zeta_{-}\Laplace r^{-}=\Tilde{f}_{r^{-}}+v_{r^{-}}\chi_0\ \text{in} \ (0,T)\times\T_L,\\
    r^{-}_{|_{t=0}}=r^{-}_{0}\ \ r^{-}_{|_{t=T}}=0,\ \ \ \ \ \ \ \  \ \ \ \,\  \text{in}\ \T_L,
    \end{array}
    \right.
\end{equation*}
given by \Cref{theorem 3.3 EEG}. As in \cite{Localcontrollabilitytotrajectoriesfornon-homogeneousincompressibleNavier-Stokesequations}, using  estimates of Theorem \ref{theorem 3.3 EEG}, we show that for any $(r_{+}^{1},r_{-}^{1})$ and $(r_{+}^{2},r_{-}^{2})$ in $\mathcal{C}_s\times\mathcal{C}_s$, we have
\begin{equation}
    \label{I 03/10/2023}
    \|(\Lambda_{s}(\Tilde{r}_{+}^{1},\Tilde{r}_{-}^{1})-\Lambda_{s}(\Tilde{r}_{+}^{2},\Tilde{r}_{-}^{2}))e^{s\vphi}\|_{L^2(L^2)}\leq C_1s^{-\frac 32} \|((\Tilde{r}_{+}^{1},\Tilde{r}_{-}^{1})-(\Tilde{r}_{+}^{2},\Tilde{r}_{-}^{2}))e^{s\vphi}\|_{L^2(L^2)},
\end{equation}
for large enough $s\geq 1$, where $C_1$ is a constant which does not depend on $s$. By equipping the space $\mathcal{C}_s\times\mathcal{C}_s$ with the following norm given for any $(r^+,r^-)$ in $\mathcal{C}_s\times\mathcal{C}_s$ by 
$$
\|(r^+,r^-)\|_{\mathcal{C}_s\times\mathcal{C}_s}:=\|(r^+,r^-)e^{s\vphi}\|_{L^2(L^2)},
$$
then $\mathcal{C}_s\times\mathcal{C}_s$ is a Banach space. Then, for $s\geq 1$ large enough, it follows from \eqref{I 03/10/2023} that the map $\Lambda_s$ is a strict contraction from $\mathcal{C}_s\times\mathcal{C}_s$ into itself. Applying Banach's fixed-point theorem, we deduce that $\Lambda_s$ has a unique fixed-point $(r^{+},r^{-})$ in $\mathcal{C}_s\times\mathcal{C}_s$. By construction, this fixed-point $(r^{+},r^{-})$ solves the controllability problem \eqref{NSK quantique Slcc**}-\eqref{NSK quantique Slcc** initial et final}. Furthermore, applying Theorem \ref{theorem 3.3 EEG} to $r^{+}$ and $r^{-}$, it follows that
\begin{align*}
s^{\frac{3}{2}}\|(r^{+},r^{-})&e^{s\vphi}\|_{L^2(L^2)}+s^{\frac 12}  \|\theta^{-1}\nabla(r^{+},r^{-})e^{s\vphi}\|_{L^2(L^2)} +\|\theta^{-\frac 32}\chi_0(v_{r^{+}},v_{r^{-}})e^{s\vphi}\|_{L^2(L^2)}\\
& \leq C_2\left( \|\theta^{-\frac 32}(f_{r^{+}},f_{r^{-}})e^{s\vphi}\|_{L^2(L^2)}+\|\theta^{-\frac 32}(r^{+},r^{-})e^{s\vphi}\|_{L^2(L^2)}+s^{\frac{1}{2}}\|(r^{+}_{0},r^{-}_{0})e^{s\vphi(0)}\|_{L^2(L^2)}\right),
\end{align*}
where $C_2$ is a positive constant, which does not depend on $s$. Since $\theta^{-\frac 32}\leq 1$, by taking $s\geq 1$ large enough to absorb the second term of the right-hand side of the above estimate, we finally obtain \begin{align}
    s^{\frac 32}\|(r^{+},r^{-}) e^{s\vphi}\|_{L^2(L^2)}+s^{\frac 12}  \|&\theta^{-1}\nabla(r^{+},r^{-})e^{s\vphi}\|_{L^2(L^2)} + \|\theta^{-\frac 32}\chi_0 (v_{r^{+}}, v_{r^{-}})  e^{s\vphi}\|_{L^2(L^2)}\nonumber \nonumber\\
     & \leq C_2\left(\|\theta^{-\frac 32}(f_{r^{+}},f_{r^{-}})e^{s\vphi}\|_{L^2(L^2)} +s^{\frac 12}\|(r^{+}_{0},r^{-}_{0})e^{s\vphi(0)}\|_{L^2}\right). \label{3.2 EGG}
\end{align}
\textit{Regularity estimates}.
 Since \eqref{minoration de theta} and $\vphi\leq\Phi$, applying Lemma \ref{proposition 3.4 EGG} to $f_{r^{+}}(r^{+},r^{-})$ and $f_{r^{-}}(r^{+},r^{-})$ such that $f_{r^{+}}(r^{+},r^{-})e^{3s\Phi/4}, f_{r^{-}}(r^{+},r^{-})e^{3s\Phi/4}\in L^2(H^1)$ and $(r^{+}_{0},r^{-}_{0})\in H^2\times H^2$, we deduce from Lemma \ref{proposition 3.4 EGG}, Item $1$ and $2$, that  $(r^+,r^-)e^{3s\Phi/4}$ belongs to $L^2(H^3)\times L^2(H^3)$. Each application of Lemma \ref{proposition 3.4 EGG} comes with estimates which, together, yield Estimate \eqref{3.4 EGG}.
\end{proof}

\section{Controllability of (\ref{NSK quantique Slcc** Jordan}): The non-diagonalizable case}\label{Sec-NonDiag}
In this section we are interested in the controllability of System \eqref{NSK quantique Slcc** Jordan}.
\begin{lem}
\label{Control Slcc** Jordan}
Let $T>0$. Let $(r^{+}_{0},r^{-}_{0})\in H^2(\T_L)\times H^2(\T_L)$. There exist a positive constant $C$ and a real number $s_0\geq 1$ such that for all $s\geq s_0$, for all $f_{r^{+}}$ and $f_{r^{-}}$ in $L^2(L^2)$ such that 
$$
\|(\theta^{-\frac{3}{2}}f_{r^{+}},\theta^{-1}f_{r^{-}})e^{s\vphi}\|_{L^2(L^2)}<+\infty
$$
and
$$
(f_{r^{+}},f_{r^{-}})e^{s\Phi}\in L^2(0,T;H^1(\T_L)),
$$
there exists a controlled trajectory $(r^{+},r^{-})$ solving \eqref{NSK quantique Slcc** Jordan} 
and satisfying the following estimate
\begin{align}
    \|(r^{+},r^{-}) e^{3s\Phi/4}&\|_{L^2(H^3)}+\|\chi_0(v_{r^{+}},v_{r^{-}})e^{3s\Phi/4}\|_{L^2(H^2)}\nonumber\\ 
    & \leq C\left(\|(f_{r^{+}},f_{r^{-}})e^{s\Phi}\|_{L^2(H^1)}+\|(r^{+}_{0},r^{-}_{0})e^{s\Phi(0)}\|_{H^2}\right).\label{3.4 EGG Jordan}
\end{align}
\end{lem}
\begin{proof}
In this proof the constant implied by the symbol $\lesssim$ is independent from the parameter $s$. Let us introduce the following two functional spaces
$$
\mathcal{C}^{-}_{s}:=\enstq{r\in L^2(0,T;L^2(\T_L))}{re^{s\vphi}\in L^2(0,T;L^2(\T_L))}
$$
and
$$
\mathcal{C}_{s}^{+}:=\enstq{r\in L^2(0,T;H^2(\T_L))}{\theta^{\frac 12}re^{s\vphi},\theta^{-\frac 32}\Laplace re^{s\vphi}\in L^2(0,T;L^2(\T_L))}
$$
which we equip respectively with the norms
$$
\|r\|_{\mathcal{C}^{-}_{s}}:=s^{\frac 32}\|re^{s\vphi}\|_{L^2(L^2)}
$$
and
$$
\|r\|_{\mathcal{C}_{s}^{+}}:=s^{\frac{5}{2}}\|\theta^{\frac 12} re^{s\vphi}\|_{L^2(L^2)}+s^{                                                        \frac 12}\|\theta^{-\frac 32}\Laplace re^{s\vphi}\|_{L^2(L^2)}.
$$
Endowed with these norms, $\mathcal{C}_{s}^{+}$ and $\mathcal{C}_{s}^{-}$ are Hilbert (hence Banach) spaces.
For $\Tilde{r}^{+}$ in $\mathcal{C}_{s}^{+}$ and $\Tilde{r}^{-}$ in $\mathcal{C}_{s}^{-}$, we introduce 
\begin{equation*}
    \left \{
\begin{array}{lcr}
     \Tilde{f}_{r^{+}}:=\Tilde{f}_{r^{+}}(\Tilde{r}^{+},\Tilde{r}^{-})=f_{r^{+}}+\beta_1\Tilde{r}^{+}+\beta_3\Tilde{r}^{-},\\
     
    \Tilde{f}_{r^{-}}:=\Tilde{f}_{r^{-}}(\Tilde{r}^{-},\Tilde{r}^{-})=f_{r^{-}}+\beta_2\Tilde{r}^{+}+\beta_4\Tilde{r}^{-}-\kappa_{\star}\Laplace \Tilde{r}^{+}.
    
\end{array}
\right.
\end{equation*}
Using Theorem \ref{theorem 3.3 M} for the equation on $r^{-}$ and Theorem \ref{theorem 3.3 EEG} for the equation on $r^{+}$, one can define a map $\Lambda_s$ on $\mathcal{C}_{s}^{+}\times\mathcal{C}_{s}^{-}$ which to a data $(\Tilde{r}^{+},\Tilde{r}^{-})$ in $\mathcal{C}_{s}^{+}\times\mathcal{C}_{s}^{-}$ associates $(r^{+},r^{-})$ where $r^{+}$ and $r^{-}$ are respectively solutions of the controlled problem
\begin{equation*}
    \left \{
\begin{array}{lcr}
     \partial_t r^{+}-\zeta\Laplace r^{+}=\Tilde{f}_{r^{+}}+v_{r^{+}}\chi_0\ \ \ \ \ \ \, \ \ \ \ \qquad \ \ \ \ \, \text{in}\ \ (0,T)\times\T_L,\\
     \partial_t r^{-}-\zeta\Laplace r^{-}=\Tilde{f}_{r^{-}}+v_{r^{-}}\chi_0\ \ \ \ \ \ \ \ \, \ \ \ \qquad \ \ \ \,\text{in}\ \ (0,T)\times\T_L,\\
     (r^{+},r^{-})_{|_{t=0}}=(r^{+}_{0},r^{-}_{0}),\ (r^{+},r^{-})_{|_{t=T}}=(0,0)\ \text{in}\ \ \T_L
    
\end{array}
\right.
\end{equation*}
given by Theorem \ref{theorem 3.3 M} and Theorem \ref{theorem 3.3 EEG}. Let $(\Tilde{r}^{+}_{1},\Tilde{r}^{-}_{1})$ and $(\Tilde{r}^{+}_{2},\Tilde{r}^{-}_{2})$ in $\mathcal{C}_{s}^{+}\times\mathcal{C}_{s}^{-}$. We set $(R^{+},R^{-}):=\Lambda_s(\Tilde{r}^{+}_{1},\Tilde{r}^{-}_{1})-\Lambda_s(\Tilde{r}^{+}_{2},\Tilde{r}^{-}_{2})$, $\Tilde{f}_{R^{+}}:=\Tilde{f}_{r^{+}}(\Tilde{r}^{+}_{1},\Tilde{r}^{-}_{1})-\Tilde{f}_{r^{+}}(\Tilde{r}^{+}_{2},\Tilde{r}^{-}_{2})$ and $\Tilde{f}_{R^{-}}:=\Tilde{f}_{r^{-}}(\Tilde{r}^{+}_{1},\Tilde{r}^{-}_{1})-\Tilde{f}_{r^{-}}(\Tilde{r}^{+}_{2},\Tilde{r}^{-}_{2})$  so that $(R^{+},R^{-})$ is a solution of the following control problem
\begin{equation*}
    \left \{
\begin{array}{lcr}
     \partial_t R^{+}-\zeta\Laplace R^{+}=\Tilde{f}_{R^{+}}
     +v_{R^{+}}\chi_0\ \ \ \ \ \ \ \ \ \, \ \, \, \, \ \ \ \ \, \ \ \, \ \, \text{in}\ \ (0,T)\times\T_L,\\
     \partial_t R^{-}-\zeta\Laplace R^{-}=\Tilde{f}_{R^{-}}+v_{R^{-}}\chi_0\ \ \ \ \ \ \ \, \ \ \ \, \ \qquad \ \ \text{in}\ \ (0,T)\times\T_L,\\
     (R^{+},R^{-})_{|_{t=0}}=(0,0),\ (R^{+},R^{-})_{|_{t=T}}=(0,0)\ \ \ \, \text{in}\ \ \T_L.
    
\end{array}
\right.
\end{equation*}
From Theorem \ref{theorem 3.3 M} and Theorem \ref{theorem 3.3 EEG}, we deduce that
\begin{align}
s^{\frac 32}\|\theta^{\frac 12}R^{+}e^{s\vphi}\|_{L^2(L^2)}+&s^{-\frac 12}\|\theta^{-\frac 32}\Laplace R^{+}e^{s\vphi}\|_{L^2(L^2)}\lesssim \|\theta^{-1}\Tilde{R}^{+}e^{s\vphi}\|_{L^2(L^2)}+\|\theta^{-1}\Tilde{R}^{-}e^{s\vphi}\|_{L^2(L^2)}\label{estime theta decale avec laplace}
\end{align}
and 
\begin{align}
    s^{\frac 32}\|R^{-}e^{s\vphi}\|_{L^2(L^2)}
    &\lesssim\|\theta^{-\frac 32}\Tilde{R}^{-}e^{s\vphi}\|_{L^2(L^2)}+\|\theta^{-\frac 32}\Tilde{R}^{+}e^{s\vphi}\|_{L^2(L^2)}+\|\theta^{-\frac 32}\Laplace\Tilde{R}^{+} e^{s\vphi}\|_{L^2(L^2)}.\label{estime a l ancienne}   
\end{align}
Multiplying \eqref{estime theta decale avec laplace} by $s$ and combining the resulting estimate with \eqref{estime a l ancienne}, and then, using that $\theta\geq 
1$ and $s\geq 1$, we get
\begin{align}
   s^{\frac 52}\|\theta^{\frac 12}R^{+}&e^{s\vphi}\|_{L^2(L^2)}+s^{\frac 12}\|\theta^{-\frac 32}\Laplace R^{+}e^{s\vphi}\|_{L^2(L^2)}+s^{\frac{3}{2}}\|R^{-}e^{s\vphi}\|_{L^2(L^2)}\nonumber\\
   &\lesssim s\|\theta^{\frac 12}\Tilde{R}^{+}e^{s\vphi}\|_{L^2(L^2)}+s\|\Tilde{R}^{-}e^{s\vphi}\|_{L^2(L^2)}+\|\theta^{-\frac 32}\Laplace\Tilde{R}^{+}e^{s\vphi}\|_{L^2(L^2)}\nonumber\\
   &\lesssim s^{-\frac{3}{2}}\left(s^{\frac 52}\|\theta^{\frac 12}\Tilde{R}^{+}e^{s\vphi}\|_{L^2(L^2)}\right)+s^{-\frac 12}\left(s^{\frac 32}\|\Tilde{R}^{-}e^{s\vphi}\|_{L^2(L^2)}\right)+s^{-\frac 12}\left(s^{\frac 12}\|\theta^{-\frac 32}\Laplace\Tilde{R}^{+}e^{s\vphi}\|_{L^2(L^2)}\right)\nonumber\\
   &\lesssim s^{-\frac 12}\left(s^{\frac 52}\|\theta^{\frac 12}\Tilde{R}^{+}e^{s\vphi}\|_{L^2(L^2)}+s^{\frac 32}\|\Tilde{R}^{-}e^{s\vphi}\|_{L^2(L^2)}+s^{\frac 12}\|\theta^{-\frac 32}\Laplace\Tilde{R}^{+}e^{s\vphi}\|_{L^2(L^2)}\right).\label{Laplace estime}
\end{align}
Then, \eqref{Laplace estime} can be rewritten as follows
$$
\|\Lambda_s(\Tilde{r}^{+}_{1},\Tilde{r}^{-}_{1})-\Lambda_s(\Tilde{r}^{+}_{2},\Tilde{r}^{-}_{2})\|_{\mathcal{C}_{s}^{+}\times\mathcal{C}_{s}^{-}}\leq Cs^{-\frac 12}\|(\Tilde{r}^{+}_{1},\Tilde{r}^{-}_{1})-(\Tilde{r}^{+}_{2},\Tilde{r}^{-}_{2})\|_{\mathcal{C}^{+}_{s}\times\mathcal{C}^{-}_{s}}, 
$$
where $C$ is a positive constant that does not depend on $s\geq s_0$. From the Banach fixed-point theorem, we deduce that for $s$ large enough, $\Lambda_s$ admits a unique fixed-point in $\mathcal{C}_{s}^{+}\times\mathcal{C}_{s}^{-}$. Let $(r^{+},r^{-})\in \mathcal{C}_{s}^{+}\times\mathcal{C}_{s}^{-}$ be the fixed-point of $\Lambda_s$ and let $(v_{r^{+}},v_{r^{-}})$ be the associated control.

\noindent\textit{Regularity estimates.} Recall that $(f_{r^{+}},f_{r^{-}})e^{s\Phi}\in L^2(H^1)$. We have $r^{-}e^{3s\Phi/4}\in L^2(L^2)$ and $r^{+}e^{3s\Phi/4}\in L^2(H^1)$, hence $f_{r^{+}}(r^{+},r^{-})e^{3s\Phi/4}\in L^2(L^2)$. Lemma \ref{proposition 3.4 EGG} thus implies that $r^{+}e^{3s\Phi/4}\in L^2(H^2)$. It follows that $f_{r^{-}}(r^{+},r^{-})e^{3s\Phi/4}\in L^2(L^2)$. Then, from Lemma \ref{proposition 3.4 EGG version Jordan}, we get that $r^{-}e^{3s\Phi/4}\in L^2(H^2)$. Accordingly, $f_{r^{+}}(r^{+},r^{-})e^{3s\Phi/4}\in L^2(H^1)$ and using again Lemma \ref{proposition 3.4 EGG}, we get $r^{+}e^{3s\Phi/4}\in L^2(H^3)$. We thus deduce that $f_{r^{-}}e^{3s\Phi/4}\in L^2(H^1)$, and finally $r^{-} e^{3s\Phi/4}\in L^2(H^3)$ from Lemma \ref{proposition 3.4 EGG version Jordan}. Each of the application of Lemma 2.3 and Lemma 2.5 comes with estimates, which directly yield \eqref{3.4 EGG Jordan}.
\end{proof}

\section{Proof of Theorem \ref{theorem 4.1 EGG}}\label{Sec-Proof-Cont}


\begin{proof}[Proof of Theorem \ref{theorem 4.1 EGG}]
\textit{\underline{Step 1: Observability of \eqref{NSK quantique Sloxc*}.}} Let us consider $s_0\geq 1$ large enough so that Lemma \ref{control Slcc**} and \ref{Control Slcc** Jordan} hold. We will recover the observability of \eqref{NSK quantique Sloxc*} from the controllability of \eqref{NSK quantique Slcc**} in the diagonalizable case obtained in Lemma \ref{control Slcc**} and of \eqref{NSK quantique Slcc** Jordan} in the non-diagonalizable case obtained in Lemma \ref{Control Slcc** Jordan}.\\ 
Let us first focus on the diagonalizable case. Let $(y^+, y^-)$ be a solution of \eqref{NSK quantique Slcc*}. By definition of the dual norm, we have
\begin{align}
    \|&(y^{+},y^{-})  e^{-s_0\Phi}\|_{L^2(H^{-1})}  +\|(y^{+}(0),y^{-}(0))e^{-s_0\Phi(0)}\|_{H^{-2}}\nonumber\\
     &=\!\sup_{\substack{\|(f_{r^{+}},f_{r^{-}})e^{s_0\Phi}\|_{L^2(H^1)}\leq 1\\ \|(r^{+}_{0},r^{-}_{0})e^{s_0\Phi(0)}\|_{H^2}\leq 1}}\!\lbrace\Re(\langle (f_{r^{+}},f_{r^{-}}),(y^{+},y^{-}) \rangle_{L^2(H^1),L^2(H^{-1})})+\Re(\langle(r^{+}_0,r^{-}_{0}),(y^{+}(0),y^{-}(0))\rangle_{H^{2},H^{-2}})\rbrace \label{definition norm dual}.
\end{align}
Now, for $(r^{+}_{0},r^{-}_{0})\in H^{2}(\T_L)$ and $(f_{r^{+}},f_{r^{-}})\in L^2(H^1)$, such that $(f_{r^{+}},f_{r^{-}})e^{s_0\Phi}\in L^2(H^1)$, we can associate the controlled  trajectory $(r^+,r^-)$ of \eqref{NSK quantique Slcc**} with corresponding controls $(v^+,v^-)$ given by Lemma \ref{control Slcc**}, and we obtain (recall that $(y^+, y^-)$ be a solution of \eqref{NSK quantique Slcc*} with source term $(g^+, g^-)$)
\begin{align*}
    \Re(\langle (f_{r^{+}}&,f_{r^{-}}),(y^{+},y^{-}) \rangle_{L^2(H^1),L^2(H^{-1})})+\Re(\langle(r^{+}_{0},r^{-}_{0}),(y^{+}(0),y^{-}(0))\rangle_{H^{2},H^{-2}}) \\
    &=\Re(\langle (g_{y^{+}},g_{y^{-}}), (r^{+},r^{-}) \rangle_{L^2(H^{-3}),L^2(H^{3})})+\Re(\langle (y^{+},y^{-}),\chi_0(v_{r^{+}},v_{r^{-}})\rangle_{L^2(H^{-1}), L^2(H^{1})}).
\end{align*}
 Consequently, using \eqref{3.4 EGG} and \eqref{definition norm dual}, we get
\begin{align}
    \|(y^{+},y^{-})e^{-s_0\Phi}&\|_{L^2(H^{-1})}+\|(y^{+}(0),y^{-}(0)) e^{-s_0\Phi(0)}\|_{H^{-2}} \nonumber\\
    & \lesssim \|(g_{y^{+}},g_{y^{-}})e^{-\frac{3s_0\Phi}{4}}\|_{L^2(H^{-3})} +\|\chi_0(y^{+},y^{-})e^{-\frac{3s_0\Phi}{4}}\|_{L^2(H^{-1})}.\label{inegalite dualite NSK quantique Slcc*}
\end{align}
In order to obtain observability for \eqref{NSK quantique Sloxc*} from \eqref{inegalite dualite NSK quantique Slcc*}, we simply remind that solutions $(y^+,y^-)$ of \eqref{NSK quantique Slcc*} correspond to solutions $(\sigma, q)$ of \eqref{NSK quantique Sloxc*} through the transform
\begin{equation*}
\begin{pmatrix}
\sigma \\
q 
\end{pmatrix}
 : = 
Q^{-1}\begin{pmatrix}
y^{+} \\
y^{-} 
\end{pmatrix}
\ \ \ \text{and}\ \ \ 
\begin{pmatrix}
g_{\sigma} \\
\divergence (g_{z}) 
\end{pmatrix}
: = 
Q^{-1}\begin{pmatrix}
g_{y^{+}} \\
g_{y^{-}} 
\end{pmatrix},
\end{equation*}
where the matrix $Q$ is the one in \eqref{Def-Q}.
When $^t\! A$ is not diagonalizable, the same strategy applies line to line, based on the duality between the control result in Lemma \ref{Control Slcc** Jordan}  for the system \eqref{NSK quantique Slcc** Jordan} and the observability of \eqref{NSK quantique Slcc* Jordan}, and the correspondence between the solutions $(y^+,y^-)$ of system \eqref{NSK quantique Slcc* Jordan} and  the solutions $(\sigma, q)$ of system \eqref{NSK quantique Sloxc*} through the matrix $R$ in \eqref{Def-R} by 
\begin{equation*}
\begin{pmatrix}
\sigma \\
q 
\end{pmatrix}
:=
R^{-1}\begin{pmatrix}
y^{+} \\
y^{-} 
\end{pmatrix}
\ \ \ \text{and}\ \ \ 
\begin{pmatrix}
g_{\sigma} \\
\divergence (g_{z}) 
\end{pmatrix},
=
R^{-1}\begin{pmatrix}
g_{y^{+}} \\
g_{y^{-}} 
\end{pmatrix}.
\end{equation*}
In both cases, we have obtained that solutions  $(\sigma,q)$ of \eqref{NSK quantique Sloxc*} satisfies,
\begin{align}
    \|(\sigma,q)e^{-s_0\Phi}&\|_{L^2(H^{-1})\times L^2(H^{-1})}+\|(\sigma(0),q(0)) e^{-s_0\Phi(0)}\|_{H^{-2}\times H^{-2}} \nonumber\\
    & \lesssim \|(g_{\sigma},g_{z})e^{-\frac{3s_0\Phi}{4}}\|_{L^2(H^{-3})\times L^2(H^{-2})} +\|\chi_0(\sigma,q)e^{-\frac{3s_0\Phi}{4}}\|_{L^2(H^{-1})\times L^2(H^{-1})},\label{changement inconnu dans inegalite de dualite}
\end{align}
that is the observability estimate for \eqref{NSK quantique Sloxc*}.

\noindent \textit{\underline{Step 2: Observability of \eqref{NSK quantique Sl*}}.} Let us rewrite the equation on $z$ in \eqref{NSK quantique Sl*} as 
\begin{equation*}
     -\partial_t z-\mu_{\star}\Laplace z=g_z+\nabla \sigma+(\mu_{\star}+\nu_{\star})\nabla q\ \ \text{in}\ (0,T)\times\T_L.
\end{equation*}
To recover the estimate on $z$, we use the duality with the following controllability problem for the heat equation
\begin{equation}
    \label{equation de le chaleur dual de la seconde equation de NSK quantique Sl*}
    \left \{
\begin{array}{lcr}
     \partial_t y - \mu_{\star} \Laplace y=\Tilde{f}_y+v_{y}\chi_0\ \ \ \text{in}\ (0,T)\times\T_{L},\\
     y_{|_{t=0}}=y_0\ \ \text{and}\ \ y_{|_{t=T}}=0\ \ \text{in}\ \T_L.
\end{array}
\right.
\end{equation}
Replacing $s$ by $4s_0/3$ in Lemma \ref{proposition 3.4 EGG}, Items $1$ and $2$, we define a map $\Xi$ which to two functions $y_0$ in $H^1(\T_L)$ and $\Tilde{f}_y$ such that $\Tilde{f}_y e^{\frac{4s_0\Phi}{3}}$ belongs in $L^2 (L^2)$, associates the solution $(y,\chi_0 v_y)$ of \eqref{equation de le chaleur dual de la seconde equation de NSK quantique Sl*} satisfying 
$$
\|ye^{s_0\Phi}\|_{L^2(H^2)}+\|\chi_0 v_y e^{s_0\Phi}\|_{L^2(L^2)}\lesssim \|\Tilde{f}_y e^{\frac{4s_0\Phi}{3}}\|_{L^2(L^2)}+\|y_0 e^{\frac{4s_0\Phi(0)}{3}}\|_{H^1}.
$$
By duality and according to the above estimate, we get
\begin{align*}
    \|ze^{-\frac{4s_0\Phi}{3}}&\|_{L^2 (L^2)}+\|z(0)e^{-\frac{4s_0\Phi(0)}{3}}\|_{H^{-1}}\\ &=\sup_{\substack{\|\Tilde{f}_ye^{\frac{4s_0\Phi}{3}}\|_{L^2(L^2)}\leq 1 \\ \|y_0e^{\frac{4s_0\Phi(0)}{3}}\|_{H^1}\leq 1}}\!\lbrace \Re\langle (\Tilde{f}_y,y_0),(z,z(0)) \rangle_{L^2(L^2)\times H^2,L^2(L^2)\times H^{-1}}\rbrace\\
    &=\sup_{\substack{\|\Tilde{f}_ye^{\frac{4s_0\Phi}{3}}\|_{L^2(L^2)}\leq 1 \\ \|y_0e^{\frac{4s_0\Phi(0)}{3}}\|_{H^1}\leq 1}}\lbrace \Re\langle \Xi(\Tilde{f}_y,y_0),(g_z+\nabla\sigma+(\nu_{\star}+\mu_{\star})\nabla q,z) \rangle_{L^2(H^2)\times L^2(L^2),L^2(H^{-2})\times L^2(L^2)}\rbrace\\
    &\lesssim \|(g_z+\nabla\sigma+(\mu_{\star}+\nu_{\star})\nabla q)e^{-s_0\Phi}\|_{L^2(H^{-2})}+\|\chi_0 ze^{-s_0\Phi(0)}\|_{L^2(L^2)}\\
    &\lesssim \|g_ze^{-\frac{3s_0\Phi}{4}}\|_{L^2(H^{-2})}+\|(\sigma,q)e^{-s_0\Phi}\|_{L^2(H^{-1})}+\|\chi_0 ze^{-s_0\Phi(0)}\|_{L^2(L^2)}.
\end{align*}
Then, according to \eqref{changement inconnu dans inegalite de dualite} it follows that
\begin{align*}
    \|ze^{-\frac{4s_0\Phi}{3}}&\|_{L^2 (L^2)}+\|z(0)e^{-\frac{4s_0\Phi(0)}{3}}\|_{H^{-1}}\\
    &\lesssim \|(g_{\sigma},g_{z})e^{-\frac{3s_0\Phi}{4}}\|_{L^2(H^{-3})\times L^{2}(H^{-2})}+\|\chi_0(\sigma,q)e^{-\frac{3s_0\Phi}{4}}\|_{L^2(H^{-1)}\times L^2(H^{-1})}
    +\|\chi_0 z e^{-\frac{3s_0\Phi}{4}}\|_{L^2(L^2)}.
\end{align*}
As $\chi=1$ on $supp(\chi_0)$, we have $\chi_0\chi=\chi_0$ and $\chi_0\!\divergence(z)=\chi_0\!\divergence(\chi z)$. Therefore, using that the multiplication by $\chi_0$ maps $H^{-1}$ to itself, we get
\begin{align*}
\|\chi_0 q e^{-\frac{3s_0\Phi}{4}}\|_{L^2(H^{-1})}\lesssim\|\chi ze^{-\frac{3s_0\Phi}{4}}\|_{L^2(L^2)}
\end{align*}
and combining the above estimate with \eqref{changement inconnu dans inegalite de dualite}, we obtain the following observability estimate for solutions $(\sigma, z)$ of \eqref{NSK quantique Sl*}
\begin{align*}
    \|\sigma e^{-s_0\Phi}\|_{L^2(H^{-1})}+\|&\sigma(0)e^{-s_0\Phi(0)}\|_{H^{-2}}+\|ze^{-\frac{4s_0\Phi}{3}}\|_{L^2(L^2)}+\|z(0)e^{-\frac{4s_0\Phi(0)}{3}}\|_{H^{-1}}\\
    &\lesssim \|(g_{\sigma},g_z)e^{-\frac{3s_0\Phi}{4}}\|_{L^2(H^{-3})\times L^2(H^{-2})}+\|\chi (\sigma,z)e^{-\frac{3s_0\Phi}{4}}\|_{L^2(H^{-1})\times L^2(L^2)}.
\end{align*}
\textit{\underline{Step 3: Conclusion.}} Since solutions $(\sigma, z)$ of \eqref{NSK quantique Sl*}
 satisfy the above observability estimate, we again argue by duality to deduce that System \eqref{NSK quantique Sl} is controllable and that the following estimate holds
\begin{align*}
\|(a,u)e^{\frac{3s_0\Phi}{4}}\|_{L^2(H^3)\times L^2(H^2)}&+\|\chi(v_a,v_u)e^{\frac{3s_0\Phi}{4}}\|_{L^2(H^1)\times L^2(L^2)} \\
&\lesssim\|(f_ae^{s_0\Phi},f_ue^{\frac{4s_0\Phi}{3}})\|_{L^2(H^1)\times L^2(L^2)}+\|(a_0e^{s_0\Phi(0)},u_0e^{\frac{4s_0\Phi(0)}{3}})\|_{H^2\times H^1}\\
&\lesssim\|(f_a,f_u)e^{\frac{4s_0\Phi}{3}})\|_{L^2(H^1)\times L^2(L^2)}+\|(a_0,u_0)e^{\frac{4s_0\Phi(0)}{3}})\|_{H^2\times H^1}.
\end{align*}
Then, it remains to establish the regularity estimate \eqref{4.2 EGG}. We perform the regularity estimate on the equation satisfied by $(a,u)e^{\frac{2s_0\Phi}{3}}$, that induces a small loss in the parameter $s_0$, which is reflected in the fact that we estimate $(a,u)e^{\frac{2s_0\Phi}{3}}$ instead of $(a,u)e^{\frac{3s_0\Phi}{4}}$ to apply the above estimate, and \eqref{4.2 EGG} follows.

Finally, it can be easily checked that the above control process, based on duality arguments, provides a linear operator $\mathcal{G}$, which, to any initial conditions $(a_0,u_0) \in H^2 \times H^1$ and source terms $(f_a,f_u)$ such that $(f_a,f_u)e^{\frac{4s_0\Phi}{3}} \in L^2(H^1)\times L^2(L^2)$, provides a controlled trajectory $(a,u)$ in $L^2(H^3)\times L^2(H^2)$ for \eqref{NSK quantique Sl}, as claimed in Theorem \ref{theorem 4.1 EGG}.
\end{proof}

\section{Proof of Theorem \ref{Controle intern de NSK quantique sur le tore}}\label{Sec-Proof-Main}
\begin{proof}[Proof of Theorem \ref{Controle intern de NSK quantique sur le tore}]
In order to prove the controllability of System \eqref{NSK quantique control}, we will perform a fixed-point argument.\\ 
We start by fixing the parameter $s_0$ so that Theorem \ref{theorem 4.1 EGG} holds.

We then define the space $\mathbf{X}\times\mathbf{Y}$, where $\mathbf{X}:=L^2(H^3)\cap L^{\infty}(H^2)\cap H^1(H^1)$ and $\mathbf{Y}:=L^2(H^2)\cap L^{\infty}(H^1)\cap H^1(L^2)$, and the following closed subset
$$
B_R:=\enstq{(\Tilde{a},\Tilde{u})\in \mathbf{X}\times\mathbf{Y}}{\|(\Tilde{a},\Tilde{u})e^{\frac{2s_0\Phi}{3}}\|_{\mathbf{X}\times\mathbf{Y}}+\|\partial_t(\Tilde{a},\Tilde{u})e^{\frac{2s_0\Phi}{3}}\|_{L^2(H^1)\times L^2(L^2)}\leq R},
$$
where $R$ is a positive real number which will be chosen later. Let $(a_0,u_0)\in H^3(\T_L)\times H^2(\T_L)$. Our goal is to find a fixed-point of the map given as follows
\begin{equation}
    \label{Definition de l'application pour point fixe}
    \mathcal{F}(\Tilde{a},\Tilde{u}):=\mathcal{G}\left(a_0,u_0,f_a(\Tilde{a},\Tilde{u}), f_u(\Tilde{a},\Tilde{u})\right),
\end{equation}
for $(\Tilde{a},\Tilde{u})$ in $B_R$ and, $f_a(\Tilde{a},\Tilde{u})$ and $f_u(\Tilde{a},\Tilde{u})$ are given (recall \eqref{definition non linearite}) by
\begin{equation*}
\left \{
\begin{array}{lcr}
  f_{a}(\Tilde{a},\Tilde{u}) :=-\Tilde{u}\cdot \nabla \Tilde{a}, \\
  f_{u}(\Tilde{a},\Tilde{u}) :=f_{u}^{1}(\Tilde{a},\Tilde{u})+f_{u}^{2}(\Tilde{a},\Tilde{u})+f_{u}^{3}(\Tilde{a},\Tilde{u})+f_{u}^{4}(\Tilde{a},\Tilde{a})+f_{u}^{5}(\Tilde{a},\Tilde{a}).
\end{array}
\right.
\end{equation*}
with
\begin{equation*}
\begin{cases}
f_{u}^{1}(\Tilde{a},\Tilde{u}):=-(\Tilde{a}+1)\Tilde{u}\cdot\nabla \Tilde{u},\\
\displaystyle f^{2}_{u}(\Tilde{a},\Tilde{u}):=\divergence\left(2\underline{\mu}(\Tilde{a})D_S(\Tilde{u}))+\nabla(\underline{\nu}(\Tilde{a})\divergence{\Tilde{u}})\right),\\
\displaystyle f_{u}^{3}(\Tilde{a},\Tilde{u}):=\partial_t\Tilde{u}\Tilde{a},\\
\displaystyle f_{u}^{4}(\Tilde{a},\Tilde{a}):= \underline{P}'(\Tilde{a})\nabla \Tilde{a},\\
\displaystyle f_{u}^{5}(\Tilde{a},\Tilde{a}):=(\Tilde{a}+1)\nabla\Big(\underline{\kappa}(\Tilde{a}) \Laplace \Tilde{a}+\nabla\underline{\kappa}(\Tilde{a})\cdot\nabla \Tilde{a}\Big).
\end{cases}
\end{equation*}
For this purpose, we prove that:
\begin{enumerate}
    \item for $R>0$ small enough, $\mathcal{F}$ is well-defined on $B_R$;
    \item for $R>0$ and $\delta>0$ small enough, we have $\mathcal{F}(B_R)\subset B_R$;
    \item for $R>0$ and $\delta>0$ small enough, $\mathcal{F}$ is a strict contraction from $B_R$ to $B_R$.
\end{enumerate}
We will next conclude by the application of Banach Picard fixed-point theorem.
In all that follows, the constant implied by the symbol $\lesssim$ is assumed to be independent from $R$, the parameter $\delta$ in \eqref{condition de petitesse pour donne initial tore}.

\noindent\textit{\underline{Step 1: $\mathcal{F}$ is well-defined on $B_R$ for all $R \in (0,R_\star]$, with $R_\star$ small enough.}} We begin by showing the following lemma.
\begin{lem}
\label{Estimation des terms nonlinear lemme}
Let $s_0\geq 1$ as in Theorem \ref{theorem 4.1 EGG}. There exist positive real numbers $R_{\star}$ and $C>0$ such that for any $R\in(0,R_{\star}]$, such that for all $(\Tilde{a},\Tilde{u})\in 
B_R$ the quantities $f_{a}(\Tilde{a},\Tilde{u})$ and $f_{u}(\Tilde{a},\Tilde{u})$ are well-defined and
\begin{equation}
\label{Estimation des terms nonlinear}
\|(f_a(\Tilde{a},\Tilde{u}), f_u(\Tilde{a},\Tilde{u}))e^{\frac{4s_0\Phi}{3}})\|_{L^2(H^1)\times L^2(L^2)}\leq C R^{2}.
\end{equation}
\end{lem}
Estimate of the nonlinear terms are based on the following classical lemma, whose proof is left to the reader as it is an adaptation of \cite[Lemma 4.10.2 p.134]{SemilinearSchrodingerEquations}.
\begin{lem}\label{I 31/08/2023}
Let $\ell>\frac{d}{2}$ be an integer and $\eta$ a positive real number. Let $F$ be a function in $\mathcal{C}^{\ell}([-\eta,\eta])$ such that $F(0)=0$. Then there exists a constant $C$ such that for any $u$ and $v$ in $H^\ell(\T_L)$ satisfy such that $\|u\|_{L^{\infty}}+\|v\|_{L^{\infty}}+\|u\|_{H^\ell} \leq \eta$, we have
$$
\|F(u)\|_{H^\ell}\leq C\|u\|_{H^\ell}\ \ \ \text{and}\ \ \ \|F(u)-F(v)\|_{H^\ell}\leq C\|u-v\|_{H^\ell}.
$$
\end{lem}
\begin{proof}[Proof of Lemma \ref{Estimation des terms nonlinear lemme}.]
Let us choose $R_{\star}>0$ such that if we denote by $K$ the constant of the Sobolev embedding $H^2(\T_L)\hookrightarrow L^{\infty}(\T_L)$, then
\begin{equation}
\label{I 30/08/2023}
KR_{\star}<\frac{\eta}{\rho_{\star}},
\end{equation}
where $\eta$ is given by {\bf (H2)}. In this case, if $(\Tilde{a},\Tilde{u})$ belongs in $B_{R_{\star}}$, then we have
\begin{equation}
\label{II 30/08/2023}
\|\Tilde{a}\|_{L^{\infty}([0,T]\times\T_L)}\leq K\|\Tilde{a}\|_{L^{\infty}(H^2)}<\frac{\eta}{\rho_{\star}},
\end{equation}
Accordingly, since  $\mu$ and $\nu$, and $\kappa$ and $P$ are respectively $\mathcal{C}^{2}$ and $\mathcal{C}^{3}$ in a neighborhood of $\rho_{\star}$ we can apply Lemma \ref{Estimation des terms nonlinear lemme} for any elements $\Tilde{a}$ such that $(\Tilde{a},\Tilde{u})$ belongs to $B_{R_{\star}}$, for some $\Tilde{u}$ in $\mathbf{Y}$. Let $R\in]0,R_{\star}]$ and $(\Tilde{a},\Tilde{u})$ in $B_R$.

We will repeatedly use that the weight function $\Phi$ depends only on the time variable and for $d\in\{1,2,3\}$ that the product is continuous from $H^1(\T_L)\times H^2(\T_L)$ to $H^1(\T_L)$ and from $H^2(\T_L)\times H^2(\T_L)$ to $H^2(\T_L)$. We will also repeatedly use that $4/3 = 2/3 + 2/3$. 
\\
\textit{Estimate on $f_a$.} We have 
$$
    \|f_a(\Tilde{a},\Tilde{u})e^{\frac{4s_0\Phi}{3}}\|_{L^2(H^1)}\lesssim \|\Tilde{u}e^{\frac{2s_0\Phi}{3}}\|_{L^{\infty}(H^1)}\|\nabla\Tilde{a}e^{\frac{2s_0\Phi}{3}}\|_{L^{2}(H^2)}
    \lesssim \|\Tilde{u}e^{\frac{2s_0\Phi}{3}}\|_{\mathbf{Y}}\|\Tilde{a}e^{\frac{2s_0\Phi}{3}}\|_{\mathbf{X}}.
$$
\textit{Estimate on $f^{1}_{u}$.} We have, according to \eqref{II 30/08/2023}
    \begin{align*}
        \|f_{u}^{1}(\Tilde{a},\Tilde{u})e^{\frac{4s_0\Phi}{3}}\|_{L^2(L^2)}\lesssim \|(\Tilde{a}+1)\|_{L^{\infty}(L^{\infty})}\|\Tilde{u}e^{\frac{2s_0\Phi}{3}}\|_{L^{2}(L^{\infty})}\|\nabla\Tilde{u}e^{\frac{2s_0\Phi}{3}}\|_{L^{\infty}(L^2)}\lesssim \|\Tilde{u}e^{\frac{2s_0\Phi}{3}}\|_{\mathbf{Y}}^{2},
    \end{align*}
\textit{Estimate on $f^{2}_{u}$.} We have
    \begin{align*}
        \|f_{u}^{2}(\Tilde{a}
        ,\Tilde{u})e^{\frac{4s_0\Phi}{3}}\|_{L^{2}(L^2)}
        &\lesssim \|\left(\divergence(2\underline{\mu}(\Tilde{a})D_S(\Tilde{u}))+\nabla(\underline{\nu}(\Tilde{a})\divergence{(\Tilde{u})})\right)e^{\frac{4 s_{0}\Phi}{3}}\|_{L^2(L^2)},
    \end{align*}
Moreover, we have
$$
\|\divergence(\underline{\mu}(\Tilde{a})D_{S}(\Tilde{u}))e^{\frac{4s_0\Phi}{3}}\|_{L^{2}(L^2)}\lesssim\|\underline{\mu}(\Tilde{a})e^{\frac{2s_0\Phi}{3}}\|_{L^{2}(H^2)}\|\Tilde{u}e^{\frac{2s_0\Phi}{3}}\|_{L^{2}(H^{2})}
$$
and
$$
\|\nabla(\underline{\nu}(\Tilde{a})\divergence(\Tilde{u}))e^{\frac{4s_0\Phi}{3}}\|_{L^{2}(L^2)}\lesssim\|\underline{\nu}(\Tilde{a})e^{\frac{2s_0\Phi}{3}}\|_{L^{\infty}(H^2)}\|\Tilde{u}e^{\frac{2s_0\Phi}{3}}\|_{L^{2}(H^{2})}.
$$
Furthermore, according to \eqref{II 30/08/2023} and Hypothesis {\bf (H2)}, we deduce from Lemma \ref{I 31/08/2023} that 
$$
\|\underline{\mu}(\Tilde{a})e^{\frac{2s_0\Phi}{3}}\|_{L^{\infty}(H^2)}\lesssim\|\Tilde{a}e^{\frac{2s_0\Phi}{3}}\|_{L^{\infty}(H^2)}\ \ \ \text{and}\ \ \ 
\|\underline{\nu}(\Tilde{a})e^{\frac{2s_0\Phi}{3}}\|_{L^{\infty}(H^2)}\lesssim\|\Tilde{a}e^{\frac{2s_0\Phi}{3}}\|_{L^{\infty}(H^2)}.
$$
Then we deduce that
$$
\|f_{u}^{2}(\Tilde{a},\Tilde{u})e^{\frac{4s_0\Phi}{3}}\|_{L^{2}(L^2)}\lesssim\|(\Tilde{a},\Tilde{u})e^{\frac{2s_0\Phi}{3}}\|_{\mathbf{X}\times\mathbf{Y}}^{2}.
$$
\textit{Estimate on $f_{u}^{3}$.} We first have
\begin{align*}
    \|f^{3}_{u}(\Tilde{a},\Tilde{u})e^{\frac{4s_0\Phi}{3}}\|_{L^{2}(L^2)}&\lesssim \|\Tilde{a}e^{\frac{2s_0\Phi}{3}}\|_{L^{\infty}(L^{\infty})}\|\partial_t\Tilde{u}e^{\frac{2s_0\Phi}{3}}\|_{L^{2}(L^2)}.\\
\end{align*}
Then, we have
$$
\|f^{3}_{u}(\Tilde{a},\Tilde{u})e^{\frac{4s_0\Phi}{3}}\|_{L^{2}(L^2)}\lesssim R^2.
$$
\textit{Estimate on $f_{u}^{4}$.} We have

    \begin{align*}
        \|f_{u}^{4}(\Tilde{a},\Tilde{a})e^{\frac{4s_0\Phi}{3}}\|_{L^2(L^2)} \lesssim\|\underline{P'}(\Tilde{a})e^{\frac{2s_0\Phi}{3}}\|_{L^{\infty}(H^2)}\|\Tilde{a}e^{\frac{2s_0\Phi}{3}}\|_{L^2(H^1)}.
    \end{align*}
 Thus, according to Lemma \ref{I 31/08/2023} applying to $\underline{P'}$, we conclude that
$$
\|f^{4}_{u}(\Tilde{a},\Tilde{a})e^{\frac{4s_0\Phi}{3}}\|_{L^{2}(H^1)}\lesssim\|(\Tilde{a},\Tilde{u})e^{\frac{2s_0\Phi}{3}}\|_{\mathbf{X}\times\mathbf{Y}}^{2}.
$$
\textit{Estimate on $f^{5}_{u}$.} According to \eqref{II 30/08/2023}, we have
\begin{align*}
    \|f_{u}^{5}(\Tilde{a},\Tilde{a})e^{\frac{4s_0\Phi}{3}}\|_{L^2(L^2)}&\lesssim \|\Tilde{a}+1\|_{L^{\infty}(L^{\infty})}\|\underline{\kappa}(\Tilde{a})\Laplace\Tilde{a}e^{\frac{4s_0\Phi}{3}}\|_{L^2(H^1)}+\|\Tilde{a}+1\|_{L^{\infty}(L^{\infty})}\|\nabla\underline{\kappa}(\Tilde{a})\cdot\nabla\Tilde{a}e^{\frac{4s_0\Phi}{3}}\|_{L^2(H^1)}\\
    &\lesssim\|\underline{\kappa}(\Tilde{a})e^{\frac{2s_0\Phi}{3}}\|_{L^{\infty}(H^2)}\|\Tilde{a}e^{\frac{2s_0\Phi}{3}}\|_{L^2(H^3)}+\|\underline{\kappa}(\Tilde{a})e^{\frac{2s_0\Phi}{3}}\|_{L^{\infty}(H^2)}\|\Tilde{a}e^{\frac{2s_0\Phi}{3}}\|_{L^2(H^3)}\\
    &\lesssim\|\underline{\kappa}(\Tilde{a})e^{\frac{2s_0\Phi}{3}}\|_{L^{\infty}(H^2)}\|\Tilde{a}e^{\frac{2s_0\Phi}{3}}\|_{\mathbf{X}}.
\end{align*}
In view of Hypothesis {\bf (H2)} and Lemma \ref{I 31/08/2023}, it follows that
$$
\|\underline{\kappa}(\Tilde{a})e^{\frac{2s_0\Phi}{3}}\|_{L^{\infty}(H^2)}\lesssim\|\Tilde{a}e^{\frac{2s_0\Phi}{3}}\|_{L^{\infty}(H^2)}.
$$
Then, we obtain that
$$
\|f^{5}_{u}(\Tilde{a},\Tilde{a})e^{\frac{4s_0\Phi}{3}}\|_{L^2(L^2)}\lesssim\|(\Tilde{a},\Tilde{u})e^{\frac{2s_0\Phi}{3}}\|_{\mathbf{X}\times\mathbf{Y}}^{2}.
$$
Combining all the above estimates, we conclude that for all $(\Tilde{a},\Tilde{u})\in B_R$
\begin{equation*}
\label{estimation bilinear over fa}
    \|f_{a}(\Tilde{a},\Tilde{u})e^{\frac{4s_0\Phi}{3}}\|_{L^{2}(H^1)}+\|f_{u}(\Tilde{a},\Tilde{u})e^{\frac{4s_0\Phi}{3}}\|_{L^{2}(L^2)}\leq CR^2,
\end{equation*}
where $C$ is a positive constant independent of $R$, which concludes the proof of Lemma \ref{Estimation des terms nonlinear lemme}.
\end{proof}
Let $R \in (0,R_{\star}]$. From Lemma \ref{Estimation des terms nonlinear lemme}, we deduce that if $(\Tilde{a},\Tilde{u})$ belongs to $B_{R}$, then $f_a(\Tilde{a},\Tilde{u})e^{\frac{4s_0\Phi}{3}}\in L^2(H^1)$ and $f_u(\Tilde{a},\Tilde{
u})e^{\frac{4s_0\Phi}{3}}\in L^2(L^2)$. Since $(a_0,u_0)$ belongs to $H^2(\T_L)\times H^1(\T_L)$, this shows, by using Estimate \eqref{4.2 EGG} of Theorem \ref{theorem 4.1 EGG}, that for all $(\Tilde{a},\Tilde{u})$ in $B_R$ the definition of $\mathcal{F}$ by \eqref{Definition de l'application pour point fixe} is meaningful. Furthermore, according to \eqref{4.2 EGG} and \eqref{Estimation des terms nonlinear}, it follows that for all $(\Tilde{a},\Tilde{u})$ in $B_R$ we have
\begin{equation}
    \label{Borne sur F avec R}
    \|\partial_t\mathcal{F}(\Tilde{a},\Tilde{u})e^{\frac{2s_0\Phi}{3}}\|_{L^2(H^1)\times L^2(L^2)}+\|\mathcal{F}(\Tilde{a},\Tilde{u})e^{\frac{2s_0\Phi}{3}}\|_{\mathbf{X}\times\mathbf{Y}}\leq C\left(R^2+\|e^{s_0\Phi(0)}(a_0,u_0)\|_{H^2\times H^1}\right).
\end{equation}
\underline{\textit{Step 2 : For $R>0$ and $\delta>0$ small enough, $\mathcal{F}(B_R)\subset B_R$.}} From the estimate \eqref{Borne sur F avec R} and the smallness assumption \eqref{condition de petitesse pour donne initial tore}, we obtain
$$
\|\partial_t\mathcal{F}(\Tilde{a},\Tilde{u})e^{\frac{2s_0\Phi}{3}}\|_{L^2(H^1)\times L^2(L^2)}+\|\mathcal{F}(\Tilde{a},\Tilde{u})e^{\frac{2s_0\Phi}{3}}\|_{\mathbf{X}\times\mathbf{Y}}\leq \Tilde{C}(R^2+\delta),
$$
where $\Tilde{C}:=Ce^{s_0\Phi(0)}$. Then, setting $R_0:=\displaystyle \min\left(\frac{1}{\Tilde{C}},R_{\star}\right)$, for all $R \in (0,R_0)$, there exists a positive real number $\delta_R$, given by 
\begin{equation}
    \label{definition de delta R}
    \delta_R:=\frac{R}{\Tilde{C}}-R^2,
\end{equation}
such that, if 
\begin{equation*}
    \|(a_0,u_0)\|_{H^2\times H^1}\leq \delta_R,
\end{equation*}
for all $(\tilde a, \tilde u)$ in $B_R$, we have the bound 
\begin{equation*}
    \|\partial_t\mathcal{F}(\Tilde{a},\Tilde{u})e^{\frac{2s_0\Phi}{3}}\|_{L^2(H^1)\times L^2(L^2)}+\|\mathcal{F}(\Tilde{a},\Tilde{u})\|_{\mathbf{X}\times\mathbf{Y}}\leq R.
\end{equation*}
This shows that for all $R\in (0,R_0]$, if $\delta \leq \delta_R$, then $\mathcal{F}(B_R)\subset B_R$. From now on, for any $R\in (0,R_0]$, the smallness of the initial data parameter $\delta$ in \eqref{condition de petitesse pour donne initial tore} will be automatically $\delta_R$.
\\
\underline{\textit{Step 3: $\mathcal{F}$ is a strict contraction from $(B_R,d)$} into itself.} On $B_R$, we consider the distance $d$ given by
$$
d(V,W):=\|(V-W)e^{\frac{2s_0\Phi}{3}}\|_{\mathbf{X}\times\mathbf{Y}}+\|\partial_t(V-W)e^{\frac{2s_0\Phi}{3}}\|_{L^2(H^1)\times L^2(L^2)}\ \ \ \ \ \ \left(V, W\in B_R\right).
$$
\\
This step is based on the following lemma.
\begin{lem}
\label{Estimation des terms nonlinear lemme contraction}
Let $s_0\geq 1$ as in Theorem \ref{theorem 4.1 EGG}. There exist $C>0$ such that for any $R\in(0,R_{0}]$, for all $(\Tilde{a}_1,\Tilde{u}_1)$ and $(\Tilde{a}_2,\Tilde{u}_2)$ in 
$B_R$ we have
\begin{equation*}
\label{Estimation des terms nonlinear contraction}
    d\!\left(\mathcal{F}(\Tilde{a}_1,\Tilde{u}_1),\mathcal{F}(\Tilde{a}_2,\Tilde{u}_2)\right)\leq CRd\!\left((\Tilde{a}_1,\Tilde{u}_1),(\Tilde{a}_2,\Tilde{u}_2)\right).
\end{equation*}
\end{lem}
The proof is based on Lemma \ref{I 31/08/2023} and follows the same lines as the one of Lemma \ref{Estimation des terms nonlinear lemme}. 
\begin{proof}[Proof of Lemma \ref{Estimation des terms nonlinear lemme contraction}.]
Let $(\Tilde{a}_1,\Tilde{u}_1)$ and $(\Tilde{a}_2,\Tilde{u}_2)$ be two elements of $B_R$. As in the proof of Lemma \ref{Estimation des terms nonlinear lemme}, we use systematically that $4/3=2/3+2/3$ and the continuity of the product from $H^1(\T_L)\times H^{2}(\T_L)$ to $H^{1}(\T_L)$ and from $H^{2}(\T_L)\times H^{2}(\T_L)$ to $H^2(\T_L)$.
\\
\textit{Estimate on $f_a$.} We can write
$$
f_a(\Tilde{a}_1,\Tilde{u}_1)-f_a(\Tilde{a}_2,\Tilde{u}_2)=f_a(\Tilde{a}_1,\Tilde{u}_1-\Tilde{u}_2)+f_u(\Tilde{a}_1-\Tilde{a}_2,\Tilde{u}_2).
$$
Then, as in the proof of Lemma \ref{Estimation des terms nonlinear lemme}, we have
\begin{align*}
    \|(f_a(\Tilde{a}_1,\Tilde{u}_1)-f_a(\Tilde{a}_2,\Tilde{u}_2))e^{\frac{4s_0\Phi}{3}}\|_{L^2(H^1)}
    &\lesssim Rd\!\left((\Tilde{a}_1,\Tilde{u}_1),(\Tilde{a}_2,\Tilde{u}_2)\right).
\end{align*}
\textit{Estimate on $f_{u}^{1}$.} Let us first note that
$$
f^{1}_{u}(\Tilde{a}_1,\Tilde{u}_1)-f^{1}_{u}(\Tilde{a}_2,\Tilde{u}_2)=-(\Tilde{a}_1-\Tilde{a}_2)\Tilde{u}_1\cdot\nabla\Tilde{u}_1-\Tilde{a}_2(\Tilde{u}_1-\Tilde{u}_2)\cdot\nabla\Tilde{u}_1-\Tilde{a}_2\Tilde{u}_2\cdot\nabla(\Tilde{u}_1-\Tilde{u}_2).
$$
Then, we have
\begin{align*}
\|(f^{1}_{u}(\Tilde{u}_1,\Tilde{u}_1)-f^{1}_{u}(\Tilde{u}_2,\Tilde{u}_2))e^{\frac{4s_0\Phi}{3}}\|_{L^2(L^2)}&\lesssim \|(\Tilde{a}_1-\Tilde{a}_2)\|_{L^{\infty}(L^{\infty})}\|\Tilde{u}_1e^{\frac{2s_0\Phi}{3}}\|_{L^{2}(L^{\infty})}\|\Tilde{u}_1e^{\frac{2s_0\Phi}{3}}\|_{L^{\infty}(H^1)}\\
&\ \ +\|\Tilde{a}_2\|_{L^{\infty}(L^{\infty})}\|(\Tilde{u}_1-\Tilde{u}_{2})e^{\frac{2s_0\Phi}{3}}\|_{L^{2}(L^{\infty})}\|\Tilde{u}_1e^{\frac{2s_0\Phi}{3}}\|_{L^{\infty}(H^1)}\\
&\ \ +\|\Tilde{a}_2\|_{L^{\infty}(L^{\infty})}\|\Tilde{u}_{2}e^{\frac{2s_0\Phi}{3}}\|_{L^{2}(L^{\infty})}\|(\Tilde{u}_1-\Tilde{u}_2)e^{\frac{2s_0\Phi}{3}}\|_{L^{\infty}(H^1)}\\
&\lesssim Rd((\Tilde{a}_1,\Tilde{u}_1),(\Tilde{a}_1,\Tilde{u}_2)).
\end{align*}
\textit{Estimate on $f^{2}_{u}$.} We first remark that
\begin{align*}
f_{v}^{2}(\Tilde{a}_1,\Tilde{u}_1)-f_{v}^{2}(\Tilde{a}_2,\Tilde{u}_2)&=2\divergence((\underline{\mu}(\Tilde{a}_1)-\underline{\mu}(\Tilde{a}_2))D_S\Tilde{u}_1)+\nabla((\underline{\nu}(\Tilde{a}_1)-\underline{\nu}(\Tilde{a}_2))\divergence(\Tilde{u}_1))\\
&\ \ \ +2\divergence(\underline{\mu}(\Tilde{a}_2)D_S(\Tilde{u}_1-\Tilde{u}_2))+\nabla(\underline{\nu}(\Tilde{a}_2)\divergence(\Tilde{u}_1-\Tilde{u}_2)).
\end{align*}
We will apply Lemma \ref{I 31/08/2023} with $F=\underline{\mu}$ and $F=\underline{\nu}$ respectively according to Hypothesis {\bf (H2)} and \eqref{II 30/08/2023}. Moreover, we have
\begin{align*}
    \|[2\divergence((\underline{\mu}(\Tilde{a}_1)-\underline{\mu}(\Tilde{a}_2))&D_S\Tilde{u}_1)+\nabla((\underline{\nu}(\Tilde{a}_1)-\underline{\nu}(\Tilde{a}_2))\divergence(\Tilde{u}_1))]e^{\frac{4s_0\Phi}{3}}\|_{L^2(L^2)}\\
    &\lesssim R\|(\underline{\mu}(\Tilde{a}_1)-\underline{\mu}(\Tilde{a}_2))e^{\frac{2s_0\Phi}{3}}\|_{L^{\infty}(H^2)}+R\|(\underline{\nu}(\Tilde{a}_1)-\underline{\nu}(\Tilde{a}_2))e^{\frac{2s_0\Phi}{3}}\|_{L^{\infty}(H^2)}\\
    &\lesssim Rd((\Tilde{a}_1,\Tilde{u}_1),(\Tilde{a}_2,\Tilde{u}_2)).
\end{align*}
We also have
\begin{align*}
\|[2\divergence(\underline{\mu}(\Tilde{a}_2)D_S(\Tilde{u}_1-\Tilde{u}_2))+\nabla(\underline{\nu}(\Tilde{a}_2)\divergence(\Tilde{u}_1-\Tilde{u}_2))]e^{\frac{4s_0\Phi}{3}}\|_{L^2(L^2)}\lesssim Rd((\Tilde{a}_1,\Tilde{u}_1),(\Tilde{a}_2,\Tilde{u}_2)).
\end{align*}
We deduce that
$$
\|(f_{u}^{2}(\Tilde{a}_1,\Tilde{u}_1)-f_{u}^{2}(\Tilde{a}_2,\Tilde{u}_2))e^{\frac{4s_0\Phi}{3}}\|_{L^{2}(L^2)}\lesssim Rd((\Tilde{a}_1,\Tilde{u}_1),(\Tilde{a}_2,\Tilde{u}_2)).
$$
\textit{Estimate on $f^{3}_{u}$.} By using the bilinearity of $f_{u}^{3}$, as for $f_{a}$, we deduce similarly as for the estimate of $f^{3}_{u}$ in the proof of Lemma \ref{Estimation des terms nonlinear lemme} that
$$
\|(f_{u}^{3}(\Tilde{a}_1,\Tilde{u}_1)-f_{u}^{3}(\Tilde{a}_2,\Tilde{u}_2))e^{\frac{4s_0\Phi}{3}}\|_{L^{2}(L^2)}\lesssim Rd((\Tilde{a}_1,\Tilde{u}_1),(\Tilde{a}_2,\Tilde{u}_2)).
$$
\textit{Estimate on $f^{4}_{u}$.} We have 
\begin{align*}
f^{4}_{u}(\Tilde{a}_1,\Tilde{a}_1)-f^{4}_{u}(\Tilde{a}_2,\Tilde{a}_2)=(\underline{P'}(\Tilde{a}_1)-\underline{P'}(\Tilde{a}_2))\nabla\Tilde{a}_1+\underline{P'}(\Tilde{a}_2)\nabla(\Tilde{a}_1-\Tilde{a}_2).
\end{align*}
Then, by applying Lemma \ref{I 31/08/2023}, we get
\begin{align*}
    \|(\underline{P'}(\Tilde{a}_1)-\underline{P'}(\Tilde{a}_2))\nabla\Tilde{a}_1]e^{\frac{4s_0\Phi}{3}}\|_{L^2(L^2)}\lesssim\|(\underline{P'}(\Tilde{a}_1)-\underline{P'}(\Tilde{a}_2))e^{\frac{2s_0\Phi}{3}}\|_{L^{\infty}(L^{\infty})}\|\Tilde{a}_1e^{\frac{2s_0\Phi}{3}}\|_{L^2(H^2)}\lesssim Rd((\Tilde{a}_1,\Tilde{u}_1),(\Tilde{a}_2,\Tilde{u}_2))
\end{align*}
and 
\begin{align*}
    \|\underline{P'}(\Tilde{a}_2)\nabla(\Tilde{a}_1-\Tilde{a}_2)e^{\frac{4s_0\Phi}{3}}\|_{L^2(L^2)}&\lesssim\|\underline{P'}(\Tilde{a}_2)e^{\frac{2s_0\Phi}{3}}\|_{L^{\infty}(L^{\infty})}\|(\Tilde{a}_1-\Tilde{a}_2)e^{\frac{2s_0\Phi}{3}}\|_{L^2(H^1)}\\
    &\lesssim Rd((\Tilde{a}_1,\Tilde{u}_1),(\Tilde{a}_2,\Tilde{u}_2)).
\end{align*}
We deduce that 
$$
\|(f_{u}^{4}(\Tilde{a}_1,\Tilde{a}_1)-f_{u}^{4}(\Tilde{a}_2,\Tilde{a}_2))e^{\frac{4s_0\Phi}{3}}\|_{L^2(L^2)}\lesssim Rd((\Tilde{a}_1,\Tilde{u}_1),(\Tilde{a}_2,\Tilde{u}_2)).
$$
\textit{Estimate on $f^{5}_{u}$.} We have
\begin{align*}
    f_{u}^{5}(\Tilde{a}_1,\Tilde{a}_1)-f_{u}^{5}(\Tilde{a}_2,\Tilde{a}_2)=&(\Tilde{a}_1+1)\nabla\Big((\underline{\kappa}(\Tilde{a}_1)-\underline{\kappa}(\Tilde{a}_2)) \Laplace \Tilde{a}_1+\nabla(\underline{\kappa}(\Tilde{a}_1)-\underline{\kappa}(\Tilde{a}_2))\cdot\nabla \Tilde{a}_1\Big)\\
&+(\Tilde{a}_1+1)\nabla\Big(\underline{\kappa}(\Tilde{a}_2) \Laplace (\Tilde{a}_1-\Tilde{a}_2)+\nabla\underline{\kappa}(\Tilde{a}_2)\cdot\nabla (\Tilde{a}_1-\Tilde{a}_2)\Big)\\
&+(\Tilde{a}_1-\Tilde{a}_2)\nabla\Big(\underline{\kappa}(\Tilde{a}_2) \Laplace \Tilde{a}_2+\nabla\underline{\kappa}(\Tilde{a}_2)\cdot\nabla \Tilde{a}_2\Big).
\end{align*}
To estimate the first term of the right-hand side, we apply Lemma \ref{I 31/08/2023} to $F=\underline{\kappa}$ according to \eqref{II 30/08/2023}. We get
\begin{align*}
    \|(\Tilde{a}_1+1)\nabla\Big((\underline{\kappa}(\Tilde{a}_1)-\underline{\kappa}(\Tilde{a}_2)) &\Laplace \Tilde{a}_1+\nabla(\underline{\kappa}(\Tilde{a}_1)-\underline{\kappa}(\Tilde{a}_2))\cdot\nabla \Tilde{a}_1\Big)e^{\frac{4s_0\Phi}{3}}\|_{L^2(L^2)}\\
    &\lesssim \|\Tilde{a}_1+1\|_{L^{\infty}(L^{\infty})}\|(\underline{\kappa}(\Tilde{a}_1)-\underline{\kappa}(\Tilde{a}_2))e^{\frac{2s_0\Phi}{3}}\|_{L^{\infty}(H^2)}\|\Laplace\Tilde{a}_1e^{\frac{2s_0\Phi}{3}}\|_{L^2(H^1)}\\
    &\ \ +\|\Tilde{a}_1+1\|_{L^{\infty}(L^{\infty})}\|\nabla(\underline{\kappa}(\Tilde{a}_1)-\underline{\kappa}(\Tilde{a}_2))e^{\frac{2s_0\Phi}{3}}\|_{L^{\infty}(H^1)}\|\nabla\Tilde{a}_1e^{\frac{2s_0\Phi}{3}}\|_{L^2(H^2)}\\
    &\lesssim\|(\underline{\kappa}(\Tilde{a}_1)-\underline{\kappa}(\Tilde{a}_2))e^{\frac{2s_0\Phi}{3}}\|_{L^{\infty}(H^2)}\|\Tilde{a}_1e^{\frac{2s_0\Phi}{3}}\|_{L^2(H^3)}\\
    &\lesssim R\|(\Tilde{a}_1-\Tilde{a}_2)e^{\frac{2s_0\Phi}{3}}\|_{L^{\infty}(H^2)}\\  
    &\lesssim Rd((\Tilde{a}_1,\Tilde{u}_1),(\Tilde{a}_2,\Tilde{u}_2)).
    \end{align*}
Similarly, we also have
$$
\|(\Tilde{a}_1+1)\nabla\Big(\underline{\kappa}(\Tilde{a}_2)\Laplace (\Tilde{a}_1-\Tilde{a}_2)+\nabla\underline{\kappa}(\Tilde{a}_2)\cdot\nabla (\Tilde{a}_1-\Tilde{a}_2)\Big)e^{\frac{4s_0\Phi}{3}}\|_{L^2(L^2)}\lesssim Rd((\Tilde{a}_1,\Tilde{u}_1),(\Tilde{a}_2,\Tilde{u}_2))
$$
and
$$
\|(\Tilde{a}_1-\Tilde{a}_2)\nabla\Big(\underline{\kappa}(\Tilde{a}_2) \Laplace \Tilde{a}_2+\nabla\underline{\kappa}(\Tilde{a}_2)\cdot\nabla \Tilde{a}_2\Big)e^{\frac{4s_0\Phi}{3}}\|_{L^2(L^2)}\lesssim Rd((\Tilde{a}_1,\Tilde{u}_1),(\Tilde{a}_2,\Tilde{u}_2)).
$$
Then we obtain that
$$
\|(f_{u}^{5}(\Tilde{a}_1,\Tilde{a}_1)-f_{u}^{5}(\Tilde{a}_2,\Tilde{a}_2))e^{\frac{4s_0\Phi}{3}}\|_{L^2(L^2)}\lesssim (R+R^2)d((\Tilde{a}_1,\Tilde{u}_1),(\Tilde{a}_2,\Tilde{u}_2)).
$$
This concludes the proof of Lemma \ref{Estimation des terms nonlinear lemme contraction}.
\end{proof}
According to Lemma \ref{Estimation des terms nonlinear lemme contraction}, we choose 
$$
R:=\frac{1}{2}\min\left(\frac{1}{ C},R_{0}\right)\ \ \ \text{and}\ \ \ \ \delta:=\delta_R,
$$
so that if $(a_0, u_0)$ satisfies
$$
\|(a_0,u_0)\|_{H^2\times H^1}\leq \delta, 
$$
the map $\mathcal{F}$ maps $B_R$ into itself by Step 2 and is contractive on $B_R$ (with the topology induced by the distance $d$ for which $B_R$ is complete) by Lemma \ref{Estimation des terms nonlinear lemme contraction}.
\\
\underline{\textit{Step 4: Conclusion}.} We conclude from the Banach fixed-point theorem that $\mathcal{F}$ admit a fixed-point $(a,u)$ in $B_R$. Moreover, it follows from Theorem \ref{theorem 4.1 EGG} and Lemma \ref{Estimation des terms nonlinear lemme} that $(a,u)$ have the wanted regularity. Finally, according to \eqref{I 30/08/2023} and \eqref{II 30/08/2023}, it follows from Hypothesis {\bf{(H2)}} that $\rho(t,x)>\rho_\star-\eta>0$ for any $(t,x)\in(0,T)\times\T_L$.
\end{proof}
\appendix
\section{Proof of the Carleman estimate}\label{App}
In this appendix we are interested in establishing the following Carleman estimate
\begin{lem}
\label{lemme carleman appendice}
Let $T>0$. Let us consider a complex number $\zeta$ such that $\Re(\zeta)>0$. There exist three positive constants $C$, $s_0\geq1$ and $\lambda_0\geq 1$, large enough, such that for all smooth function $w$ on $[0,T]\times\T_L$ and for all $s\geq s_0$ and $\lambda\geq\lambda_0$, we have
\begin{align*}
s^{\frac{3}{2}}\lambda^{2}\|\xi^{\frac{3}{2}}we^{-s\vphi}\|_{L^2(L^2)}+s^{\frac 12}\lambda\|\xi^{\frac{1}{2}}\nabla & we^{-s\vphi}\|_{L^2(L^2)}  +s\lambda^{\frac{3}{2}}e^{7\lambda}\|w(0)e^{-s\vphi(0)}\|_{L^2}\\
& \leq C\left(\|(-\partial_t-\zeta\Laplace)we^{-s\vphi}\|_{L^2(L^2)}+s^{\frac 32}\lambda^{2}\|\xi^{\frac 32}\chi_0 we^{-s\vphi}\|_{L^2(L^2)}\right),
\end{align*}
where we have set
$$
\xi(t,x):=\theta(t)e^{\lambda\psi(x)}.
$$
\end{lem}
\begin{proof}
Let us set $\zeta=\alpha+i\beta$, where $\alpha$ and $\beta$ are real numbers with $\alpha>0$. Let $w$ be a smooth complex valued function on $[0,T]\times\T_L$ and set 
$$
f:=-\partial_tw-\zeta\Laplace{w}.
$$
We shall deal with the function
$$
\mathbf{w}:=e^{-s\vphi}w.
$$
According to the definition of $\theta$, $\mathbf{w}$ satisfies
$$
\mathbf{w}(T,x)=0\ \ \ \text{and}\ \ \ \nabla \mathbf{w}(T,x)=0,\ \ \ x\in\T_L.
$$
Let us define the conjugate of $-\zeta\partial_t-\Laplace$ by
$$
P_{\vphi}:=e^{-s\vphi}(-\partial_t-\zeta\Laplace)e^{s\vphi}.
$$
Then 
$$
P_{\vphi}\mathbf{w}=-\partial_t\mathbf{w}-s\partial_t\vphi \mathbf{w}-\zeta\Laplace \mathbf{w}-2s\zeta\nabla\vphi\cdot\nabla \mathbf{w}-s^2\zeta |\nabla\vphi|^2\mathbf{w}-s\zeta\Laplace\vphi \mathbf{w}
$$
and 
$$
e^{-s\vphi}f=P_{\vphi}\mathbf{w}.
$$
Inspired by the strategy to prove Carleman estimate (see \cite{NullControllabilityoftheComplexGinzburgLandauEquation}), we now define quantities $P_1\mathbf{w}$ and $P_2\mathbf{w}$ from the symmetric and  antisymmetric part of $P_{\vphi}$ by setting

$$
P_1\mathbf{w}:=\frac{1}{2}(P_{\vphi}+P^{*}_{\vphi})\mathbf{w}=-\alpha(\Laplace \mathbf{w}+s^2|\nabla\vphi|^2\mathbf{w})-i\beta(2s\nabla\vphi\cdot\nabla\mathbf{w}+s\Laplace\vphi\mathbf{w})-s\partial_t\vphi\mathbf{w}
$$
and 
$$
P_2\mathbf{w}:=\frac{1}{2}(P_{\vphi}-P^{*}_{\vphi})\mathbf{w}=-\partial_t \mathbf{w}-i\beta(\Laplace \mathbf{w}+s^2|\nabla\vphi|^2\mathbf{w})-\alpha(2s\nabla\vphi\cdot\nabla\mathbf{w}+s\Laplace\vphi\mathbf{w}),
$$
so that
$$
P_{\vphi}\mathbf{w}=P_1\mathbf{w}+P_2\mathbf{w}.
$$
Since $P_1\mathbf{w}+P_2\mathbf{w}=e^{-s\vphi}f$, we get
\begin{align}
\iint_{[0,T]\times\T_L}|P_1\mathbf{w}|^2+&\iint_{[0,T]\times\T_L}|P_2\mathbf{w}|^2+2\Re\left(\iint_{[0,T]\times\T_L}P_1\mathbf{w}\overline{P_2 \mathbf{w}}\right)\leq \iint_{[0,T]\times\T_L}e^{-2s\vphi}|f|^2.\label{A.8 BEG}
\end{align}
The main part of the proof $L^2$-Carleman estimate consists to estimate from below the real part of the scalar product of $P_1\mathbf{w}$ with $P_2\mathbf{w}$. We begin by setting
\begin{equation}
\label{cross product}
\Re\left(\iint_{[0,T]\times\T_L}P_1\mathbf{w}\overline{P_2 \mathbf{w}}\right)=\sum_{1\leq k,l\leq 3}\Re(I_{k,l}),
\end{equation}
where $I_{k,l}$ is the scalar product of the $k$-th term of $P_1\mathbf{w}$ with the $l$-th term of $P_2\mathbf{w}$.
Note that by $L$-periodicity, all the boundary terms generated from integration by parts with respect to the space variable vanish. \\
\textit{\underline{Step 1: Computation of the scalar product.}} Let us begin by remarking that
$$
\Re(I_{2,3})=\Re(I_{1,2})=0.
$$
\textit{Computation of $\Re(I_{1,3})+\Re(I_{2,2})$.} We get
\begin{align}
    \Re(I_{1,3})+\Re(I_{2,2})&=(\alpha^2+\beta^2)\Re\left(\iint_{[0,T]\times\T_L}(\Laplace\mathbf{w}+s^2|\nabla\vphi|^2\mathbf{w})(2s\nabla\vphi\cdot\nabla\overline{\mathbf{w}}+s\Laplace\vphi\overline{\mathbf{w}})\right)\nonumber\\
    &=2|\zeta|^2s\Re\left(\iint_{[0,T]\times\T_L}\Laplace\mathbf{w}\nabla\vphi\cdot\nabla\overline{\mathbf{w}}\right)-|\zeta|^2s\Re\left(\iint_{[0,T]\times\T_L}\nabla\Laplace\vphi\cdot\nabla\mathbf{w}\overline{\mathbf{w}}\right)\nonumber\\
    &\ \ \ -|\zeta|^2 s\iint_{[0,T]\times\T_L}\Laplace\vphi|\nabla\mathbf{w}|^2\nonumber\\
    &\ \ \ -|\zeta|^2s^3\iint_{[0,T]\times\T_L}\divergence(|\nabla\vphi|^2\nabla\vphi)|\mathbf{w}|^2+|\zeta|^2s^3\iint_{[0,T]\times\T_L}|\nabla\vphi|^2\Laplace\vphi|\mathbf{w}|^2\nonumber\\
    &=2|\zeta|^2s\Re\left(\iint_{[0,T]\times\T_L}\Laplace\mathbf{w}\nabla\vphi\cdot\nabla\overline{\mathbf{w}}\right)+\frac{|\zeta|^2s}{2}\iint_{[0,T]\times\T_L}\Laplace^2\vphi|\mathbf{w}|^2\nonumber\\
    &\ \ \ -|\zeta|^2s\iint_{[0,T]\times\T_L}\Laplace\vphi|\nabla\mathbf{w}|^2\nonumber\\
    &\ \ \ -|\zeta|^2s^3\iint_{[0,T]\times\T_L}\divergence(|\nabla\vphi|^2\nabla\vphi)|\mathbf{w}|^2+|\zeta|^2s^3\iint_{[0,T]\times\T_L}|\nabla\vphi|^2\Laplace\vphi|\mathbf{w}|^2\nonumber\\
    &=-2|\zeta|^2s\Re\left(\iint_{[0,T]\times\T_L}D^2\vphi(\nabla\mathbf{w},\nabla\overline{\mathbf{w}})\right)+\frac{|\zeta|^2}{2}s\iint_{[0,T]\times\T_L}\Laplace^2\vphi|\mathbf{w}|^2\nonumber\\
    &\ \ \ -|\zeta|^2s^3\iint_{[0,T]\times\T_L}\divergence(|\nabla\vphi|^2\nabla\vphi)|\mathbf{w}|^2+|\zeta|^2s^3\iint_{[0,T]\times\T_L}|\nabla\vphi|^2\Laplace\vphi|\mathbf{w}|^2.\label{III 03/05/2023}
\end{align}
\textit{Computation of $\Re(I_{1,1})$.} By integrating by parts and using that $\mathbf{w}(T)=\nabla\mathbf{w}(T)=0$, we deduce
\begin{align}
    \Re(I_{1,1})&=\alpha\Re\left(\iint_{[0,T]\times\T_L}(\Laplace\mathbf{w}+s^2|\nabla\mathbf{\vphi}|^2\mathbf{w})\partial_t\overline{\mathbf{w}}\right)\nonumber\\
    &=\alpha\Re\left(\iint_{[0,T]\times\T_L}\Laplace\mathbf{w}\partial_t\overline{\mathbf{w}}\right)-\frac{\alpha s^2}{2}\int_{\T_L}|\nabla\vphi(0)|^2|\mathbf{w}(0)|^2-\frac{\alpha s^2}{2}\iint_{[0,T]\times\T_L}\partial_t|\nabla\vphi|^2|\mathbf{w}|^2\nonumber\\
    &=\frac{\alpha}{2}\int_{\T_L}|\nabla\mathbf{w}(0)|^2-\frac{\alpha s^2}{2}\int_{\T_L}|\nabla\vphi(0)|^2|\mathbf{w}(0)|^2-\frac{\alpha s^2}{2}\iint_{[0,T]\times\T_L}\partial_t|\nabla\vphi|^2|\mathbf{w}|^2. \label{IV 03/05/2023}
\end{align}
\textit{Computation of $\Re(I_{2,1})$.} Keeping in mind that $\mathbf{w}(T)=\nabla\mathbf{w}(T)=0$ and using that for any complex number $z$ we have $\Im(z)+\Im(\overline{z})=0$, by integrating by parts, we have
\begin{align}
    \Re(I_{2,1})&=\Re\left(i\beta\iint_{[0,T]\times\T_L}(2s\nabla\vphi\cdot\nabla\mathbf{w}+s\Laplace\vphi\mathbf{w})\partial_t\overline{\mathbf{w}}\right)\nonumber\\
    &=2s\Re\left(i\beta\iint_{[0,T]\times\T_L}\nabla\vphi\cdot\nabla\mathbf{w}\partial_t\overline{\mathbf{w}}\right)-s\Re\left(i\beta\iint_{[0,T]\times\T_L}\Laplace\vphi\mathbf{w}\partial_t\overline{\mathbf{w}}\right)\nonumber\\
    &=s\Re\left(i\beta\iint_{[0,T]\times\T_L}\nabla\vphi\cdot\nabla\mathbf{w}\partial_t\overline{\mathbf{w}}\right)-s\Re\left(i\beta\iint_{[0,T]\times\T_L}\nabla\vphi\cdot\mathbf{w}\partial_t\nabla\overline{\mathbf{w}}\right)\nonumber\\
    &=-s\Re\left(i\beta\int_{\T_L}\nabla\vphi(0)\cdot\nabla \mathbf{w}(0)\overline{\mathbf{w}}(0)\right)-s\Re\left(i\beta\iint_{[0,T]\times\T_L}\partial_t\nabla\vphi\cdot\nabla\mathbf{w}\overline{\mathbf{w}}\right)\nonumber\\
    &\ \ \ -s\Re\left(i\beta\iint_{[0,T]\times\T_L}\nabla\vphi\cdot\partial_t\nabla\mathbf{w}\overline{\mathbf{w}}\right)-s\Re \left(i\beta\iint_{[0,T]\times\T_L}\nabla\vphi\cdot\partial_t\nabla\overline{\mathbf{w}}\mathbf{w}\right)\nonumber\\
    &=-s\Re\left(i\beta\int_{\T_L}\nabla\vphi(0)\cdot\nabla\mathbf{w}(0)\overline{\mathbf{w}}(0)\right)-s\Re\left(i\beta\iint_{[0,T]\times\T_L}\partial_t\nabla\vphi\cdot\nabla\mathbf{w}\overline{\mathbf{w}}\right).\label{I 05/05/2023}
\end{align}
\textit{Computation of $\Re(I_{3,1})$.} Since $\mathbf{w}(T)=0$, we obtain
\begin{equation}
\label{VI 05/05/2023}
\Re(I_{3,1})=s\Re\left(\iint_{[0,T]\times\T_L}\partial_t\vphi\mathbf{w}\partial_t\overline{\mathbf{w}}\right)=-\frac{s}{2}\int_{\T_L}\partial_t\vphi(0)|\mathbf{w}(0)|^2-\frac{s}{2}\iint_{[0,T]\times\T_L}\partial_{t}^{2}\vphi|\mathbf{w}|^2.
\end{equation}
\textit{Computation of $\Re(I_{3,2})$.} We have
\begin{align}
    \Re(I_{3,2})&=-\Re\left(i\beta\iint_{[0,T]\times\T_L}s\partial_t\vphi\mathbf{w}(\Laplace\overline{\mathbf{w}}+s^2|\nabla\vphi|^2\overline{\mathbf{w}}\right)\nonumber\\
    &=-s\Re\left(i\beta\iint_{[0,T]\times\T_L}\partial_t\vphi\cdot\Laplace\overline{\mathbf{w}}\mathbf{w}\right)\nonumber\\
    &=s\Re\left(i\beta\iint_{[0,T]\times\T_L}\partial_t\nabla\vphi\cdot\nabla\overline{\mathbf{w}}\mathbf{w}\right).\label{I 16/05/2023}
\end{align}
\textit{Computation of $\Re(I_{3,3})$.} We have
\begin{align}
    \Re(I_{3,3})&=\alpha s^2\Re\left(\iint_{[0,T]\times\T_L}(2\nabla\vphi\cdot\nabla\overline{\mathbf{w}}+\Laplace\vphi\overline{\mathbf{w}})\partial_t\vphi\mathbf{w}\right)\nonumber\\
    &=-s^2\alpha\iint_{[0,T]\times\T_L}\divergence(\nabla\vphi\partial_t\vphi)|\mathbf{w}|^2+s^2\alpha\iint_{[0,T]\times\T_L}\Laplace\vphi\partial_t\vphi|\mathbf{w}|^2\nonumber\\
    &=-\alpha s^2\iint_{[0,T]\times\T_L}\nabla\vphi\cdot\partial_t\nabla\vphi|\mathbf{w}|^2\nonumber\\
    &=-\frac{\alpha s^2}{2}\iint_{[0,T]\times\T_L}\partial_t|\nabla\vphi|^2|\mathbf{w}|^2.\label{I  11/05/2023}
\end{align}
From \eqref{I 05/05/2023} and \eqref{I 16/05/2023}, we deduce that
\begin{equation}
    \label{II 16/05/2023}
    \Re(I_{3,2})+\Re(I_{2,1})=-s\Re\left(i\beta\int_{\T_L}\nabla\vphi(0)\cdot\nabla\mathbf{w}(0)\overline{\mathbf{w}}(0)\right)+2s\Re\left(i\beta\iint_{[0,T]\times\T_L}\partial_t\nabla\vphi\cdot\nabla\overline{\mathbf{w}}\mathbf{w}\right).
\end{equation}
By combining \eqref{III 03/05/2023}, \eqref{IV 03/05/2023}, \eqref{VI 05/05/2023}, \eqref{I  11/05/2023} and \eqref{II 16/05/2023}, it follows that
\begin{align}
    &\Re\left(\langle P_1\mathbf{w},P_2\mathbf{w}\rangle\right)\nonumber\\
    &=\frac{\alpha}{2}\int_{\T_L}|\nabla\mathbf{w}(0)|^2-\frac{s}{2}\int_{\T_L}\partial_t\vphi(0)|\mathbf{w}(0)|^2-\frac{s^2\alpha}{2}\int_{\T_L}|\nabla\vphi(0)|^2|\mathbf{w}(0)|^2-s\Re\left(i\beta\int_{\T_L}\nabla\vphi(0)\cdot\nabla\mathbf{w}(0)\overline{\mathbf{w}}(0)\right)\label{s0}\\
    &\ \ \ +\frac{s|\zeta|^2}{2}\iint_{[0,T]\times\T_L}\Laplace^2\vphi|\mathbf{w}|^2+s^3|\zeta|^2\iint_{[0,T]\times\T_L}(|\nabla\vphi|^2\Laplace\vphi-\divergence(|\nabla\vphi|^2\nabla\vphi))|\mathbf{w}|^2\label{s11}\\
    &\ \ \ -s^2\alpha\iint_{[0,T]\times\T_L}\partial_t|\nabla\vphi|^2|\mathbf{w}|^2-\frac{s}{2}\iint_{[0,T]\times\T_L}\partial_{t}^{2}\vphi|\mathbf{w}|^2\label{s2}\\
    &\ \ \ +s\left(-2|\zeta|^2\Re\left(\iint_{[0,T]\times\T_L}D^{2}\vphi(\nabla\mathbf{w},\nabla\overline{\mathbf{w}})\right)+2\Re\left(i\beta\iint_{[0,T]\times\T_L}\partial_t\nabla\vphi\cdot\nabla\overline{\mathbf{w}}\mathbf{w}\right)\right).\label{s12}
\end{align}
\textit{\underline{Step 2: Lower bound of the scalar product.}} We now give a lower bound for the scalar product $\Re\langle P_1\mathbf{w},P_2\mathbf{w}\rangle$. Along the rest of the proof, we will take the parameters $s$ and $\lambda$ large enough in order to absorb lower order terms with respect to the power of these parameters. In the following, to simplify notations, we will denote by $C^{\star}$ a generic large positive constant which do not depends on $s$ and $\lambda$ and by $C_{\star}$ a generic small positive constant independent of $s$ and $\lambda$. The constants may change from line to line.\\
\textit{Lower bound of \eqref{s0}.} We have
$$
-\partial_t\vphi(0)=\frac{s\lambda^2e^{2\lambda}(\lambda e^{12\lambda}-e^{\lambda\psi})}{T_0}.
$$
Since $\psi\leq 7$, we deduce that
$$
-\partial_t\vphi(0)\geq C_{\star}s\lambda^3e^{14\lambda}.
$$
Thus we obtain that
$$
-\frac{s}{2}\int_{\T_L}\partial_t\vphi(0)|\mathbf{w}(0)|^2\geq C_{\star}s^2\lambda^3e^{14\lambda}\int_{\T_L}|\mathbf{w}(0)|^2.
$$
Besides, since $\nabla\vphi(0)=-2\lambda\nabla\psi e^{\lambda\psi}$ and $\psi\leq 7$, we deduce that
\begin{align*}
    -\frac{\alpha s^2}{2}\int_{\T_L}|\nabla\vphi(0)|^2|\mathbf{w}(0)|^2\geq -C^{\star}s^2\lambda^2e^{14\lambda}\int_{\T_L}|\mathbf{w}(0)|^2
\end{align*}
and
\begin{align}
    -s\Re\left(i\beta\int_{\T_L}\nabla\vphi(0)\cdot\nabla\mathbf{w}(0)\overline{\mathbf{w}}(0)\right)\geq -C^{\star}s^2\lambda^{\frac{5}{2}}e^{14\lambda}\int_{\T_L}|\mathbf{w}(0)|^2-\frac{C^{\star}}{\lambda^{\frac{1}{2}}}\int_{\T_L}|\nabla\mathbf{w}(0)|^2.
\end{align}
According to $\alpha>0$, we conclude that 
\begin{align}
    \frac{\alpha}{2}\int_{\T_L}&|\nabla\mathbf{w}(0)|^2-\frac{s}{2}\int_{\T_L}\partial_t\vphi(0)|\mathbf{w}(0)|^2-\frac{s^2\alpha}{2}\int_{\T_L}|\nabla\vphi(0)|^2|\mathbf{w}(0)|^2\nonumber\\
    &\geq C_{\star}\int_{\T_L}|\nabla\mathbf{w}(0)|^2+C_{\star}s^2\lambda^3 e^{14\lambda}\int_{\T_L}|\mathbf{w}(0)|^2-C^{\star}s^2\lambda^{\frac{5}{2}}e^{14\lambda}\int_{\T_L}|\mathbf{w}(0)|^2-\frac{C^{\star}}{\lambda^{\frac{1}{2}}}\int_{\T_L}|\nabla\mathbf{w}(0)|^2\nonumber\\
    &\geq C_{\star}\int_{\T_L}|\nabla\mathbf{w}(0)|^2+C_{\star}s^2\lambda^3 e^{14\lambda}\int_{\T_L}|\mathbf{w}(0)|^2.\label{I 25/06/2023}
\end{align}
\\
\textit{Lower bound of \eqref{s11}.} We first have 
\begin{align}
    \frac{s|\zeta|^2}{2}\iint_{[0,T]\times\T_L}\Laplace^2\vphi|\mathbf{w}|^2\geq -C^{\star}s\lambda^4\iint_{[0,T]\times\T_L}\xi^3|\mathbf{w}|^2.
\end{align}
Besides, we have
\begin{align*}
    -|\zeta|^2s^3\iint_{[0,T]\times\T_L}\divergence(|\nabla\vphi|^2\nabla\vphi)|\mathbf{w}|^2+&|\zeta|^2s^3\iint_{[0,T]\times\T_L}|\nabla\vphi|^2\Laplace\vphi|\mathbf{w}|^2\\
    &=-|\zeta|^2s^3\iint_{[0,T]\times\T_L}\nabla|\nabla\vphi|^2\cdot\nabla\vphi|\mathbf{w}|^2\\
    &\geq -C^{\star}s^3\lambda^3\iint_{[0,T]\times\T_L}\xi^2|\mathbf{w}|^2+|\zeta|^2s^3\lambda^4\iint_{[0,T]\times\T_L}|\nabla\psi|^4\xi^3|\mathbf{w}|^2.
\end{align*}
Moreover, since $\inf\{|\nabla\psi|\}>0$ on $\T_L\setminus\overline{\omega}$, and $|\zeta|>0$, we deduce that
$$
|\zeta|^2s^3\lambda^4\iint_{[0,T]\times\T_L}|\nabla\psi|^4\xi^3|\mathbf{w}|^2\geq C_{\star}s^3\lambda^4\iint_{[0,T]\times\T_L}\xi^3|\mathbf{w}|^2-C^{\star}s^3\lambda^4\iint_{[0,T]\times\omega}\xi^3|\mathbf{w}|^2.
$$
Thus, we have
\begin{align}
\frac{s|\zeta|^2}{2}\iint_{[0,T]\times\T_L}\Laplace^2\vphi|\mathbf{w}|^2& -|\zeta|^2s^3\iint_{[0,T]\times\T_L}\divergence(|\nabla\vphi|^2\nabla\vphi)|\mathbf{w}|^2+|\zeta|^2s^3\iint_{[0,T]\times\T_L}|\nabla\vphi|^2\Laplace\vphi|\mathbf{w}|^2\nonumber\\
 &\geq C_{\star}s^3\lambda^4\iint_{[0,T]\times\T_L}\xi^3|\mathbf{w}|^2-C^{\star}s^3\lambda^4\iint_{[0,T]\times\omega}\xi^3|\mathbf{w}|^2\label{L4}.
\end{align}
\textit{Lower bound of \eqref{s2}.} 
According to the definition of $\vphi$, we have
$$
\partial^{2}_t\vphi=\frac{\partial_{t}^{2}\theta}{\theta}\vphi.
$$
Furthermore, in views of the definition of $\theta$,  on $[0,T_0]\times\T_L$, we have
$$
0\leq \partial_{t}^{2}\theta\leq C^{\star}s^2\lambda^4e^{4\lambda}.
$$
Thus, since $\psi\geq 6$ and $\theta\geq 1$, we obtain
$$
-\partial_{t}^{2}\vphi\geq -C^{\star}s^2\lambda^5e^{16\lambda}\geq -C^{\star}s^2\lambda^3\xi^3,
$$
on $[0,T_0]\times\T_L$. On the other hand, on $[T_0,T)\times\T_L$, we have
$$
-\partial_{t}^{2}\vphi\geq -C^{\star}\lambda\xi^2.
$$
We deduce that
$$
-\frac{s}{2}\iint_{[0,T]\times\T_L}\partial_{t}^{2}\vphi|\mathbf{w}|^2\geq -C^{\star}s^{3}\lambda^3\iint_{[0,T]\times\T_L}\xi^3|\mathbf{w}|^2.
$$
Moreover, it follows from the definition of $\theta$ that 
$$
\frac{\partial_t\theta}{\theta}\leq 0\ \ \text{on}\ \ [0,T-2T_1]\ \ \text{and}\ \ \frac{|\partial_t\theta|}{\theta}\leq C^{\star}\xi\ \ \text{on}\ \ [T-2T_1,T).
$$
Then, we deduce that
$$
-\frac{\alpha s^2}{2}\iint_{[0,T]\times\T_L}\partial_t|\nabla\vphi|^2|\mathbf{w}|^2\geq -C^{\star}s^2\lambda^2\iint_{[0,T]\times\T_L}\xi^3|\mathbf{w}|^3.
$$
Then, we conclude that
\begin{equation}
\label{L3}
-\frac{s}{2}\iint_{[0,T]\times\T_L}\partial_{t}^{2}\vphi|\mathbf{w}|^2-\frac{\alpha s^2}{2}\iint_{[0,T]\times\T_L}\partial_t|\nabla\vphi|^2|\mathbf{w}|^2\geq 
-C^{\star}s^3\lambda^3\iint_{[0,T]\times\T_L}\xi^3|\mathbf{w}|^2.
\end{equation}
\textit{Lower bound of \eqref{s12}.} From the definition of $\vphi$, we deduce that
\begin{align*}
-2|\zeta|^2s\Re\left(\iint_{[0,T]\times\T_L}D^2\vphi(\nabla\mathbf{w},\nabla\overline{\mathbf{w}})\right)&=2|\zeta|^2s\lambda\Re\left(\iint_{[0,T]\times\T_L}\xi D^2\psi(\nabla\mathbf{w},\nabla\overline{\mathbf{w}})\right)\\
&\ \ \ 2|\zeta|^2\lambda^2s\iint_{[0,T]\times\T_L}\xi|\nabla\psi\cdot\nabla\mathbf{w}|^2\\
&\geq 2|\zeta|^2s\lambda\Re\left(\iint_{[0,T]\times\T_L}\xi D^2\psi(\nabla\mathbf{w},\nabla\overline{\mathbf{w}})\right)\\
&\geq -C^{\star}s\lambda\iint_{[0,T]\times\T_L}\xi|\nabla\mathbf{w}|^2.
\end{align*}
Furthermore, we have
$$
-s\Re\left(i\beta\iint_{[0,T]\times\T_L}\partial_t\nabla\vphi\cdot\nabla\overline{\mathbf{w}}\mathbf{w}\right)=s\lambda\Re\left(ib\iint_{[0,T]\times\T_L}\frac{\partial_t\theta}{\theta}\xi\nabla\psi\cdot\nabla\overline{\mathbf{w}}\mathbf{w}\right).
$$
On $[0,T-2T_1]$, we have
$$
s\lambda\xi\frac{|\partial_t\theta|}{\theta}\leq \frac{s^2\lambda^3\xi e^{2\lambda}}{T_0},
$$
thus
\begin{align*}
s\lambda\Re\left(i\beta\iint_{[0,T-2T_1]\times\T_L}\frac{\partial_t\theta}{\theta}\xi\nabla\psi\cdot\nabla\overline{\mathbf{w}}\mathbf{w}\right)&\geq -C^{\star}s^3\lambda^4e^{2\lambda}\iint_{[0,T]\times\T_L}\xi^2|\mathbf{w}|^2-C^{\star}s\lambda^2e^{2\lambda}\iint_{[0,T]\times\T_L}|\nabla\mathbf{w}|^2\\
&\geq -C^{\star}s^3\lambda^3\iint_{[0,T]\times\T_L}\xi^3|\mathbf{w}|^2 -C^{\star}s\lambda\iint_{[0,T]\times\T_L}\xi|\nabla\mathbf{w}|^2.
\end{align*}
Besides, on $[T-2T_1,T)$, we have
$$
s\lambda\xi\frac{|\partial_t\theta|}{\theta}\leq C^{\star}s\lambda\xi^2,
$$
and then
$$
s\lambda\Re\left(i\beta\iint_{[0,T]\times\T_L}\frac{|\partial_t\theta|}{\theta}\xi\nabla\psi\cdot\nabla\overline{\mathbf{w}}\mathbf{w}\right)\geq -C^{\star}s\lambda\iint_{[0,T]\times\T_L}\xi^3|\mathbf{w}|^2-C^{\star}s\lambda\iint_{[0,T]\times\T_L}\xi|\nabla\mathbf{w}|^2.
$$
We deduce that
$$
-s\Re\left(i\beta\iint_{[0,T]\times\T_L}\partial_t\nabla\vphi\cdot\nabla\overline{\mathbf{w}}\mathbf{w}\right)\geq-C^{\star}s^{3}\lambda^3\iint_{[0,T]\times\T_L}\xi^3|\mathbf{w}|^2-C^{\star}s\lambda\iint_{[0,T]\times\T_L}\xi|\nabla\mathbf{w}|^2.
$$
If we denote by $L_2$ the expression of Line \eqref{s12}, we deduce that
\begin{equation}
    \label{L2}
    L_2\geq -C^{\star}s^{3}\lambda^3\iint_{[0,T]\times\T_L}\xi^3|\mathbf{w}|^2-C^{\star}s\lambda\iint_{[0,T]\times\T_L}\xi|\nabla\mathbf{w}|^2.
\end{equation}
We deduce from \eqref{I 25/06/2023}, \eqref{L4}, \eqref{L2} and \eqref{L3}, that
\begin{align*}
    \Re\left(\langle P_1\mathbf{w},P_2\mathbf{w}\rangle\right)&\geq C_{\star}\int_{\T_L}|\nabla\mathbf{w}(0)|^2+C_{\star}s^2\lambda^3e^{14\lambda}\int_{\T_L}|\mathbf{w}(0)|^2\\
    &\ \ +C_{\star}s^3\lambda^4\iint_{[0,T]\times\T_L}\xi^3|\mathbf{w}|^2-C^{\star}s^3\lambda^4\iint_{[0,T]\times\omega}\xi^3|\mathbf{w}|^2-C^{\star}s\lambda\iint_{[0,T]\times\T_L}\xi|\nabla\mathbf{w}|^2.
\end{align*}
\textit{\underline{Step 3: Observation on $[0,T]\times\omega$.}} Using the previous estimate and \eqref{A.8 BEG}, we obtain that
\begin{align}
   \iint_{[0,T]\times\T_L} e^{-2s\vphi}|f|^2+&C^{\star}s\lambda\iint_{[0,T]\times\T_L}\xi|\nabla\mathbf{w}|^2+C^{\star}s^3\lambda^4\iint_{[0,T]\times\omega}\xi^3|\mathbf{w}|^2\nonumber\\
    &\geq C_{\star}s^3\lambda^4\iint_{[0,T]\times\T_L}\xi^3|\mathbf{w}|^2+C_{\star}s^2\lambda^3 e^{14\lambda}\int_{\T_L}|\mathbf{w}(0)|^2+C_{\star}\int_{\T_L}|\nabla\mathbf{w}(0)|^2\nonumber\\
    &\ \ +C_{\star}\iint_{[0,T]\times\T_L}|P_1\mathbf{w}|^2+C_{\star}\iint_{[0,T]\times\T_L}|P_2\mathbf{w}|^2.\label{I 28/04/2023}
\end{align}
Moreover, we have the following lemma.
\begin{lem}\label{II 28/04/2023}
For any $\lambda\geq 1$ and $s\geq 1$ large enough, we have 
\begin{equation}
\label{II 28/04/2023 eq}
s\lambda^2\iint_{[0,T]\times\T_L}\xi|\nabla\mathbf{w}|^2\leq C^{\star}s^3\lambda^4\iint_{[0,T]\times\T_L}\xi^3|\mathbf{w}|^2+C^{\star}\iint_{[0,T]\times\T_L}|P_1\mathbf{w}|^2.
\end{equation}
\end{lem}
\begin{proof}
We have
\begin{align*}
    s\lambda^2\iint_{[0,T]\times\T_L}\xi|\nabla\mathbf{w}|^2&=-s\lambda^2\Re\left(\iint_{[0,T]\times\T_L}\nabla\xi\cdot\nabla\mathbf{w}\overline{\mathbf{w}}\right)-s\lambda^2\Re\left(\iint_{[0,T]\times\T_L}\xi\Laplace\mathbf{w}\overline{\mathbf{w}}\right)\\
    &=\frac{s\lambda^2}{2}\iint_{[0,T]\T_L}\Laplace\xi|\mathbf{w}|^2-s\lambda^2\Re\left(\iint_{[0,T]\times\T_L}\xi\Laplace\mathbf{w}\overline{\mathbf{w}}\right).
\end{align*}
Then, from 
$$
-\Laplace\mathbf{w}=\frac{1}{\alpha}\left(P_1\mathbf{w}+\alpha s^2|\nabla\vphi|^2\mathbf{w}+i\beta(2s\nabla\vphi\cdot\nabla\mathbf{w}+s\lambda\Laplace\vphi\mathbf{w})+s\partial_t\vphi\mathbf{w}\right)\ \ \ \text{on}\ \ [0,T)\times\T_L,
$$
it follows that
\begin{align*}
    s\lambda^2\iint_{[0,T]\times\T_L}\xi|\mathbf{w}|^2&=\frac{s\lambda^2}{2}\iint_{[0,T]\times\T_L}\Laplace\xi|\mathbf{w}|^2+\frac{s\lambda^2}{\alpha}\Re\left(\iint_{[0,T]\times\T_L}P_1\mathbf{w}\xi\overline{\mathbf{w}}\right)\\
    &\ \ \ +s^3\lambda^2\iint_{[0,T]\times\T_L}\xi|\nabla\vphi|^2|\mathbf{w}|^2\\
    &\ \ \ +\frac{s\lambda^2}{\alpha}\Re\left(ib\iint_{[0,T]\times\T_L}(2s\nabla\vphi\cdot\nabla\mathbf{w}+s\lambda\Laplace\vphi\mathbf{w})\xi\overline{\mathbf{w}}\right)\\
    &\ \ \ +\frac{s^2\lambda^2}{\alpha}\iint_{[0,T]\times\T_L}\xi\partial_t\vphi|\mathbf{w}|^2\\
    &=\frac{s\lambda^2}{2}\iint_{[0,T]\times\T_L}\Laplace\xi|\mathbf{w}|^2+\frac{s\lambda^2}{\alpha}\Re\left(\iint_{[0,T]\times\T_L}P_1\mathbf{w}\xi\overline{\mathbf{w}}\right)\\
    &\ \ \ +s^3\lambda^2\iint_{[0,T]\times\T_L}\xi|\nabla\vphi|^2|\mathbf{w}|^2\\
    &\ \ \ +\frac{2s^2\lambda^2}{\alpha}\Re\left(ib\iint_{[0,T]\times\T_L}\nabla\vphi\cdot\nabla\mathbf{w}\xi\overline{\mathbf{w}}\right)\\
    &\ \ \ +\frac{s^2\lambda^2}{\alpha}\iint_{[0,T]\times\T_L}\xi\partial_t\vphi|\mathbf{w}|^2.
\end{align*}
Furthermore, since $\partial_t\vphi\leq 0$ on $[0,T-2T_1]\times\T_L$, we deduce that
\begin{equation}
\label{II 17/07/2023}
s^2\lambda^2\iint_{[0,T]\times\T_L}\xi\partial_t\vphi|\mathbf{w}|^2\leq s^2\lambda^2\iint_{[T-2T_1,T]\times\T_L}\xi\partial_t\vphi|\mathbf{w}|^2.
\end{equation}
Moreover, by using the definition of $\theta$ and \eqref{2.41 EGG}, we deduce that
$$
|\partial_t\vphi|\leq\frac{|\partial_t\theta|}{\theta}\vphi\leq C^{\star}\theta\vphi\leq C^{\star}\theta^2\lambda e^{12\lambda}\ \ \ \text{on}\ \ [T-2T_1,T)\times\T_L.
$$
Since $\psi\geq 6$, we obtain that
\begin{equation*}
\partial_t\vphi\leq C^{\star}\lambda\xi^2\ \ \ \text{on}\ \ [T-2T_1,T)\times\T_L.
\end{equation*}
Hence, from \eqref{II 17/07/2023}, it follows that
\begin{equation}
    \label{I 15/07/2023}
s^2\lambda^2\iint_{[0,T]\times\T_L}\xi\partial_t\vphi|\mathbf{w}|^2\leq C^{\star}s^2\lambda^3\iint_{[0,T]\times\T_L}\xi^3|\mathbf{w}|^2.
\end{equation}
On the other hand, by using the Young estimate, we get
\begin{equation}
    \label{II 15/07/2023}
    s\lambda^2\Re\left(\iint_{[0,T]\times\T_L}P_1\mathbf{w}\xi\overline{\mathbf{w}}\right)\leq C^{\star}\iint_{[0,T]\times\T_L}|P_1\mathbf{w}|^2+C^{\star}s^2\lambda^4\int_{[0,T]\times\T_L}\xi^2|\mathbf{w}|^2.
\end{equation}
Furthermore, it follows from the Young estimate that
\begin{align}
    \frac{2s^2\lambda^2}{\alpha}\Re&\left(i\beta\iint_{[0,T]\times\T_L}\nabla\vphi\cdot\nabla\mathbf{w}\xi\overline{\mathbf{w}}\right)\nonumber\\
    &\leq s^2\lambda^2|\beta|\left(\frac{1}{s(|\beta|+1)}\iint_{[0,T]\times\T_L}\xi|\nabla\mathbf{w}|^2+\frac{4(|\beta|+1)s}{\alpha^2}\iint_{[0,T]\times\T_L}|\nabla\vphi|^2\xi|\mathbf{w}|^2\right)\nonumber\\
    &\leq\frac{|\beta|}{|\beta|+1}s\lambda^2\iint_{[0,T]\times\T_L}\xi|\nabla\mathbf{w}|^2+C^{\star}s^3\lambda^4\iint_{[0,T]\times\T_L}\xi^3|\mathbf{w}|^2.\label{III 15/07/2023}
\end{align}
Then, by combining \eqref{I 15/07/2023}, \eqref{II 15/07/2023} and \eqref{III 15/07/2023}, we get
\begin{align*}
    s\lambda^2\iint_{[0,T]\times\T_L}\xi|\nabla\mathbf{w}|^2&\leq C^{\star}\iint_{[0,T]\times\T_L}|P_1\mathbf{w}|^2+C^{\star}s^3\lambda^4\iint_{[0,T]\times\T_L}\xi^3|\mathbf{w}|^2\\
    &\ \ \ +\frac{|\beta|}{|\beta|+1}s\lambda^2\iint_{[0,T]\times\T_L}\xi|\nabla\mathbf{w}|^2\\
    &\ \ \ +\frac{s\lambda^2}{2}\iint_{[0,T]\times\T_L}\Laplace\xi|\mathbf{w}|^2+s^3\lambda^2\iint_{[0,T]\times\T_L}\xi|\nabla\vphi|^2|\mathbf{w}|^2\\
    &\leq C^{\star}\iint_{[0,T]\times\T_L}|P_1\mathbf{w}|^2+C^{\star}s^3\lambda^4\iint_{[0,T]\times\T_L}\xi^3|\mathbf{w}|^2\\
    &\ \ \ +C^{\star}s\lambda^4\iint_{[0,T]\times\T_L}\xi^3|\mathbf{w}|^2+\frac{|\beta|}{|\beta|+1}s\lambda^2\iint_{[0,T]\times\T_L}\xi|\nabla\mathbf{w}|^2,
\end{align*}
that is 
\begin{align*}
    \left(1-\frac{|\beta|}{|\beta|+1}\right)s\lambda^2\iint_{[0,T]\times\T_L}\xi|\nabla\mathbf{w}|^2\leq &C^{\star}\iint_{[0,T]\times\T_L}|P_1\mathbf{w}|^2+C^{\star}s^3\lambda^4\iint_{[0,T]\times\T_L}\xi^3|\mathbf{w}|^2\\
    &\ \ \ +C^{\star}s\lambda^4\iint_{[0,T]\times\T_L}\xi^3|\mathbf{w}|^2.
\end{align*}
\end{proof}
We apply Lemma \ref{II 28/04/2023}. Then, by using \eqref{I 28/04/2023} to estimate the last term of the right-hand side of \eqref{II 28/04/2023 eq} and by absorbing the first term of the right-hand side, we deduce that
\begin{align}
 C_{\star}s\lambda^2\iint_{[0,T]\times\T_L}\xi|\nabla\mathbf{w}|^2 &\leq s\lambda^2\iint_{[0,T]\times\T_L}\xi|\nabla\mathbf{w}|^2\nonumber\\
 &\leq C^{\star}\iint_{[0,T]\times\T_L}|P_1\mathbf{w}|^2+C^{\star}s^3\lambda^4\iint_{[0,T]\times\T_L}\xi^3|\mathbf{w}|^2\nonumber\\
 &\leq C^{\star}\iint_{[0,T]\times\T_L} e^{-2s\vphi}|f|^2+C^{\star}s^3\lambda^4\iint_{[0,T]\times\omega}\xi^3|\mathbf{w}|^2\nonumber\\
 &\ \ \ +C^{\star}s\lambda\iint_{[0,T]\times\T_L}\xi|\nabla\mathbf{w}|^2.\label{I 06/06/2023}
\end{align}
By combining estimates \eqref{I 28/04/2023} and \eqref{I 06/06/2023}, we obtain 
\begin{align}
C_{\star}s^3\lambda^4\iint_{[0,T]\times\T_L}\xi^3&|\mathbf{w}|^2+C_{\star}s\lambda^2\iint_{[0,T]\times\T_L}\xi|\nabla\mathbf{w}|^2+C_{\star}s^2\lambda^3e^{14\lambda}\int_{\T_L}|\mathbf{w}(0)|^2+C_{\star}\int_{\T_L}|\nabla\mathbf{w}(0)|^2\nonumber\\
    &\leq C^{\star}\iint_{[0,T]\times\T_L}e^{-2s\vphi}|f|^2+C^{\star}s^3\lambda^4\iint_{[0,T]\times\omega}\xi^3|\mathbf{w}|^2.\label{I 29/04/2023}
\end{align}
\underline{\textit{Step 4: Observation on $[0,T]\times supp(\chi_0)$.}} Let us remark that the observation is done on $\omega\subset\{\chi_0=1\}\subset supp(\chi_0)$. Thus, we have
$$
\iint_{[0,T]\times\omega}\xi^3|\mathbf{w}|^2\leq \iint_{[0,T]\times\T_L}\chi_{0}\xi^3|\mathbf{w}|^2.
$$
Then, we deduce from \eqref{I 29/04/2023} that
\begin{align}
    s^3\lambda^4\iint_{[0,T]\times\T_L}\xi^3|\mathbf{w}|^2+&s\lambda^2\iint_{[0,T]\times\T_L}\xi|\nabla\mathbf{w}|^2+s^2\lambda^3e^{14\lambda}\int_{\T_L}|\mathbf{w}(0)|^2+\int_{\T_L}|\nabla\mathbf{w}(0)|^2\nonumber\\
    &\leq C^{\star}\left(\iint_{[0,T]\times\T_L} e^{-2s\vphi}|f|^2+s^3\lambda^4\iint_{[0,T]\times\T_L}\chi_0\xi^3|\mathbf{w}|^2\right).\label{A.50 BEG}
\end{align}
\underline{\textit{Step 5: Conclusion}.} It is enough to recover the estimate on $z$ from \eqref{A.50 BEG}. Since $w=\mathbf{w}e^{s\vphi}$, we get
$$
    |w|^2e^{-2s\varphi}=|\mathbf{w}|^2
$$
and 
$$
    |\nabla w|^2e^{-2s\vphi}\leq 2|\nabla \mathbf{w}|^2+2s^2|\nabla\vphi|^2|\mathbf{w}|^2\leq 2|\nabla \mathbf{w}|^2+2C^{\star}s^2\lambda^2\xi^2|\mathbf{w}|^2.
$$
Combining the above estimate and \eqref{A.50 BEG}, we conclude the proof of Lemma \ref{lemme carleman appendice}.
\end{proof}
\bibliographystyle{abbrv}
\bibliography{Bibliographie}

\begin{thebibliography}{10}

\bibitem{Localcontrollabilitytotrajectoriesfornon-homogeneousincompressibleNavier-Stokesequations}
M.~Badra, S.~Ervedoza, and S.~Guerrero.
\newblock Local controllability to trajectories for non-homogeneous incompressible {N}avier-{S}tokes equations.
\newblock {\em Ann. Inst. H. Poincar\'{e} C Anal. Non Lin\'{e}aire}, 33(2):529--574, 2016.

\bibitem{StructureofKortewegmodelsandstabilityofdiffuseinterfaces}
S.~Benzoni-Gavage, R.~Danchin, S.~Descombes, and D.~Jamet.
\newblock Structure of {K}orteweg models and stability of diffuse interfaces.
\newblock {\em Interfaces and Free Boundaries}, 7:371--414, 2005.

\bibitem{OnNavierStokesKortewegandEulerKortewegsystemsapplicationtoquantumfluidsmodels}
D.~Bresch, M.~Gisclon, and I.~Lacroix-Violet.
\newblock On {N}avier-{S}tokes-{K}orteweg and {E}uler-{K}orteweg systems: application to quantum fluids models.
\newblock {\em Arch. Ration. Mech. Anal.}, 233(3), 2019.

\bibitem{DerivationofviscouscorrectiontermsfortheisothermalquantumEulermodel}
S.~Brull and F.~Méhats.
\newblock Derivation of viscous correction terms for the isothermal quantum {E}uler model.
\newblock {\em ZAMM Journal of applied mathematics and mechanics: Zeitschrift für angewandte Mathematik und Mechanik}, 90:219--230, 03 2010.

\bibitem{LocalcontrollabilityofthestabilizedKuramotoSivashinskysystembyasinglecontrolactingontheheatequation}
N.~Carre\~{n}o and E.~Cerpa.
\newblock Local controllability of the stabilized {K}uramoto-{S}ivashinsky system by a single control acting on the heat equation.
\newblock {\em J. Math. Pures Appl. (9)}, 106(4):670--694, 2016.

\bibitem{SemilinearSchrodingerEquations}
T.~Cazenave.
\newblock {\em Semilinear Schrodinger Equations}.
\newblock Courant lecture notes in mathematics. American Mathematical Society, 2003.

\bibitem{GevreyanalyticityanddecayforthecompressibleNavier-Stokessystemwithcapillarity}
F.~Charve, R.~Danchin, and J.~Xu.
\newblock {Gevrey analyticity and decay for the compressible {N}avier-{S}tokes system with capillarity}.
\newblock {\em Indiana University Mathematics Journal}, 2018.

\bibitem{ControllabilityandStabilizabilityoftheLinearizedCompressibleNavier-StokesSysteminOneDimension}
S.~Chowdhury, M.~Ramaswamy, and J.~P. Raymond.
\newblock Controllability and stabilizability of the linearized compressible {N}avier-{S}tokes system in one dimension.
\newblock {\em SIAM J. Control. Optim.}, 50:2959--2987, 2012.

\bibitem{Sharpanddiffuseinterfacemethodsforphasetransitionproblemsinliquidvapourflows}
F.~Coquel, D.~Diehl, C.~Merkle, and C.~Rohde.
\newblock Sharp and diffuse interface methods for phase transition problems in liquid-vapour flows.
\newblock {\em IRMA Lect. Math. Theor. Phys.}, 7, 09 2009.

\bibitem{LocalnullcontrollabilityofthethreedimensionalNavierStokessystemwithadistributedcontrolhavingtwovanishingcomponents}
J.-M. Coron and P.~Lissy.
\newblock Local null controllability of the three-dimensional {N}avier-{S}tokes system with a distributed control having two vanishing components.
\newblock {\em Invent. Math.}, 198(3):833--880, 2014.

\bibitem{ExistenceofsolutionsforcompressiblefluidmodelsofKortewegtype}
R.~Danchin and B.~Desjardins.
\newblock Existence of solutions for compressible fluid models of {K}orteweg type.
\newblock {\em Annales de l'Institut Henri Poincaré C, Analyse non linéaire}, 18(1):97--133, 2001.

\bibitem{OntheThermodynamicsofInterstitialWorking}
J.~E. Dunn and J.~Serrin.
\newblock On the thermomechanics of interstitial working.
\newblock {\em Arch. Rational Mech. Anal.}, 88(2):95--133, 1985.

\bibitem{Indirectcontrollabilityofsomelinearparabolicsystemsofmequationswithm1controlsinvolvingcouplingtermsofzeroorfirstorde}
M.~Duprez and P.~Lissy.
\newblock Indirect controllability of some linear parabolic systems of m equations with m-1 controls involving coupling terms of zero or first order.
\newblock {\em Journal de Mathématiques Pures et Appliquées}, 106(5):905--934, 2016.

\bibitem{LocalexactcontrollabilityforthetwoandthreedimensionalcompressibleNavierStokesequations}
S.~Ervedoza, O.~Glass, and S.~Guerrero.
\newblock Local exact controllability for the two- and three-dimensional compressible {N}avier–{S}tokes equations.
\newblock {\em Communications in Partial Differential Equations}, 41:1660 -- 1691, 2015.

\bibitem{LocalExactControllabilityfortheOneDimensionalCompressibleNavierStokesEquation}
S.~Ervedoza, O.~Glass, S.~Guerrero, and J.-P. Puel.
\newblock Local exact controllability for the one-dimensional compressible {N}avier-{S}tokes equation.
\newblock {\em Arch. Ration. Mech. Anal.}, 206(1):189--238, 2012.

\bibitem{Localboundarycontrollabilitytotrajectoriesforthe1DcompressibleNavierStokesequations}
S.~Ervedoza and M.~Savel.
\newblock Local boundary controllability to trajectories for the 1{D} compressible {N}avier {S}tokes equations.
\newblock {\em ESAIM Control Optim. Calc. Var.}, 24(1):211--235, 2018.

\bibitem{Nullcontrollabilityfortheparabolicequationwithacomplexprincipalpart}
X.~Fu.
\newblock Null controllability for the parabolic equation with a complex principal part.
\newblock {\em Journal of Functional Analysis}, 257:1333--1354, 09 2009.

\bibitem{ControllabilityofEvolutionequations}
A.~V. Fursikov and O.~Y. Imanuvilov.
\newblock Controllability of evolution equations.
\newblock {\em Seoul National University}, 1996.

\bibitem{GlobalstrangsolutionfortheKortewegsystemwithquantumpressureindimensionNgeq2}
B.~Haspot.
\newblock Global strong solution for the {K}orteweg system with quantum pressure in dimension $n\geq 2$.
\newblock {\em Mathematische Annalen}, 367(1):667--700, 2017.

\bibitem{GlobalSolutionsofaHighDimensionalSystemforKortewegMaterials}
H.~Hattori and D.~Li.
\newblock Global solutions of a high dimensional system for {K}orteweg materials.
\newblock {\em Journal of Mathematical Analysis and Applications}, 198:84--97, 1996.

\bibitem{GlobalWeakSolutionstoCompressibleNavier-StokesEquationsforQuantumFluids}
A.~J{\"u}ngel.
\newblock Global weak solutions to compressible {N}avier-{S}tokes equations for quantum fluids.
\newblock {\em SIAM J. Math. Anal.}, 42:1025--1045, 2010.

\bibitem{Dissipativestructureforsymmetrichyperbolic-parabolicsystemswithKorteweg-typedispersion}
S.~Kawashima, Y.~Shibata, and J.~Xu.
\newblock Dissipative structure for symmetric hyperbolic-parabolic systems with {K}orteweg-type dispersion.
\newblock {\em Comm. Partial Differential Equations}, 47(2):378--400, 2022.

\bibitem{Somecontrollabilityresultsforlinearizedcompressiblenavierstokessystem}
D.~Maity.
\newblock Some controllability results for linearized compressible {N}avier-{S}tokes system.
\newblock {\em ESAIM: Control, Optimisation and Calculus of Variations}, 21(4):1002--1028, 2015.

\bibitem{LocalExactBoundaryControllabilityfortheCompressibleNavierStokesEquations}
N.~Molina.
\newblock Local exact boundary controllability for the compressible {N}avier-{S}tokes equations.
\newblock {\em SIAM J. Control. Optim.}, 57:2152--2184, 2019.

\bibitem{NullControllabilityoftheComplexGinzburgLandauEquation}
L.~Rosier and B.-Y. Zhang.
\newblock Null {Controllability} of the {Complex} {Ginzburg-Landau} {Equation}.
\newblock {\em Annales de l'I.H.P. Analyse non lin\'eaire}, 26(2):649--673, 2009.

\bibitem{GlobalexistenceandanalyticityofLpsolutionstothecompressiblefluidmodelofKortewegtype}
Z.~Song and J.~Xu.
\newblock Global existence and analyticity of {$L^p$} solutions to the compressible fluid model of {K}orteweg type.
\newblock {\em J. Differential Equations}, 370:101--139, 2023.

\bibitem{AnalyticregularityforNavierStokesKortewegmodelonpseudomeasurespaces}
A.~Tendani~Soler.
\newblock Analytic regularity for {N}avier-{S}tokes-{K}orteweg model on pseudo-measure spaces.
\newblock {\em Dynamics of Partial Differential Equations}, 20(1):1--21, 2023.

\bibitem{ObservationandControlforOperatorSemigroups}
M.~Tucsnak and G.~Weiss.
\newblock {\em Observation and Control for Operator Semigroups}.
\newblock Birkh{\"a}user Advanced Texts Basler Lehrb{\"u}cher. Birkh{\"a}user Basel, 2009.

\end{thebibliography}
\end{document}